\documentclass[fontsize=10pt, paper=a4, headsepline=true]{scrartcl}
\pdfoutput=1

\KOMAoption{abstract}{true}
\setkomafont{section}{\large}
\setkomafont{subsection}{\normalsize}
\setkomafont{subsubsection}{\normalsize}
\setkomafont{paragraph}{\normalsize}

\usepackage[automark]{scrlayer-scrpage}
	\clearpairofpagestyles
	\ohead[]{\pagemark}
	\ihead[]{\headmark}
	\ofoot[\pagemark]{}
\usepackage[inner=2.5cm,outer=3cm,
			top=3.5cm,bottom=3.5cm]{geometry} 
\usepackage[T1]{fontenc}
\usepackage[utf8]{inputenc}
\usepackage{lmodern}					
\usepackage[british]{babel} 
\usepackage[colorlinks=true]{hyperref}
\usepackage{csquotes}

\usepackage[%
	style=alphabetic,%
	giveninits=true,%
	doi=true,%
	url=true,%
	bibencoding=utf8,%
	backend=biber
]{biblatex}
\bibliography{moving_stokes_lit}

\AtEveryBibitem{%
	\ifentrytype{book}{%
		\clearfield{url}%
		\clearfield{urldate}%
  	}{}%
}
				
\usepackage{amssymb,amsmath,amsthm,amsfonts}
\usepackage{mathtools}
\usepackage{commath} 			
\usepackage{nicefrac}
\usepackage{array} 	
\usepackage{multirow}

\usepackage{bm}
\usepackage{dsfont}
\usepackage{calrsfs}
\usepackage{stmaryrd}
\DeclareMathAlphabet{\pazocal}{OMS}{zplm}{m}{n}
\DeclareMathAlphabet{\pazocalbf}{OMS}{cmsy}{b}{n}
\usepackage{etoolbox}
\preto\subequations{\ifhmode\unskip\fi}

\usepackage{enumitem}
	\setlist[itemize]{nosep, topsep=0pt, wide = 1em, leftmargin=*}	
\usepackage{booktabs}
\usepackage{siunitx}
\usepackage{graphicx} 
\usepackage[export]{adjustbox}
\usepackage[hyperref, dvipsnames]{xcolor}

\usepackage[caption=true, font={sf, footnotesize}]{subfig}
\setkomafont{caption}{\sffamily\footnotesize}
\setkomafont{captionlabel}{\bfseries\usekomafont{caption}}

\makeatletter
\renewcommand\paragraph{%
  \@startsection{paragraph}
    {4}
    {\z@}
    {3.25ex \@plus1ex \@minus.2ex}
    {-1em}
    {\normalfont\sectfont\normalsize\bfseries\maybe@addperiod}%
}
\newcommand{\maybe@addperiod}[1]{%
  #1\@addpunct{.}%
}
\makeatother

\numberwithin{equation}{section}
\newcommand\numberthis{\addtocounter{equation}{1}\tag{\theequation}}
\newcounter{theorems}
\numberwithin{theorems}{section}

\theoremstyle{plain}
\newtheorem{theorem}[theorems]{Theorem}
\newtheorem{lemma}[theorems]{Lemma}

\theoremstyle{definition}

\newtheorem{assumption}{Assumption}

\theoremstyle{remark} 
\newtheorem{remark}[theorems]{Remark}



\AtBeginEnvironment{tabular}{\footnotesize}

\newcommand{\R}{\mathds{R}}							
\newcommand{\N}{\mathds{N}}							
\newcommand{\PP}{\mathbb{P}}
\renewcommand{\P}[2]{\PP^{#1}(#2)}					
\newcommand{\E}{\pazocal{E}^\mathbb{P}}				

\renewcommand{\L}[2]{\pazocalbf{L}^{#1}(#2)}		
\newcommand{\Linf}[1]{\L{\infty}{#1}}				
\newcommand{\Ltwo}[1]{\pazocal{L}^2(#1)}		
\newcommand{\Ltwozero}[1]{\pazocal{L}^2_0(#1)}		
\renewcommand{\H}[3]{\pazocalbf{H}^{#1}_{#2}(#3)}	
\newcommand{\W}[2]{\pazocalbf{W}^{#1}(#2)}			

\newcommand{\dual}[3]{\langle #1, #2\rangle_{#3}}	
\newcommand{\lprod}[3]{(#1,#2)_{#3}}				
\newcommand{\llprod}[3]{\left(#1,#2\right)_{#3}}				
\newcommand{\nrm}[2]{\Vert#1\Vert_{#2}}				
\newcommand{\tripnrm}[1]{\mathopen{|\mkern-1.5mu|\mkern-1.5mu|}%
  #1\mathclose{|\mkern-1.5mu|\mkern-1.5mu|}}		
\newcommand{\tripnrmast}[2]{\tripnrm{#1}_{\ast,#2}}	
\newcommand{\tripnrmdt}[2]{\tripnrm{#1}_{\flat,#2}}

\renewcommand{\O}{\Omega}							
\newcommand{\Ot}{\widetilde{\O}}					
\newcommand{\G}{\Gamma}							
\newcommand{\Q}{\pazocal{Q}}						
\newcommand{\On}{\O^n}								
\newcommand{\Onh}{\On_h}							
\newcommand{\Ghn}{\G_h^n}							
\newcommand{\OO}{\pazocal{O}}						
\newcommand{\Od}[1]{\OO_{\delta}(#1)}				
\newcommand{\Odh}[1]{\mathchoice{\OO_{\delta_{h}}(#1)}{\OO_{\delta_{h}}(#1)}%
{\OO_{\delta_{\adjustbox{scale=0.6}{$\scriptstyle h$}}}(#1)}{}}
\newcommand{\OdT}[1]{\mathchoice{\OO^{#1}_{\delta_{h},\T}}{%
\OO^{#1}_{\delta_{h},\T}}
{\OO^{#1}_{\delta_{\adjustbox{scale=0.6}{$\scriptstyle h$}},\T}}{}}
\newcommand{\Scal}{\pazocal{S}}
\newcommand{\Spm}{\Scal^\pm}						
\newcommand{\Sp}{\Scal^{+}} 						

\newcommand{\T}{\mathcal{T}}
\newcommand{\Th}{\T_h}								
\newcommand{\Tht}{\widetilde{\T}_h}					
\newcommand{\Thn}{\T^n_{h,\delta_h}}          		
\newcommand{\ThGhn}{\T_{h,\mathchoice{\Ghn}{\Ghn}{\Ghn}{%
	\adjustbox{scale=0.7}{$\scriptstyle \G^n_h$}}}}
\newcommand{\ThGn}{\T_{h,\mathchoice{\G^n}{\G^n}{\G^n}{%
	\adjustbox{scale=0.7}{$\scriptstyle \G^n$}}}}

\newcommand{\TSpm}{\T^n_{h,\Spm}}					
\newcommand{\TSp}{\T^n_{h,\Sp}}						
\newcommand{\Fh}{\mathcal{F}_h}						
\newcommand{\Fhn}{\Fh^n}							
\newcommand{\Fhnd}{\mathcal{F}^n_{h,\delta_h}}		

\newcommand{\V}{\bm{V}}								
\newcommand{\Vnh}{\V^n_h}							
\newcommand{\Qnh}{Q^n_h}							

\renewcommand{\u}{\bm{u}}							
\let\vs\v
\renewcommand{\v}{\bm{v}}							
\newcommand{\vhl}{\v_h^\ell}						
\newcommand{\qhl}{q_h^\ell}							

\newcommand{\w}{\bm{w}}								
\newcommand{\n}{\bm{n}}								
\newcommand{\x}{\bm{x}}								
\newcommand{\f}{\bm{f}}								

\newcommand{\Eu}[1]{\mathbb{E}^{#1}}				
\newcommand{\Ep}[1]{\mathbb{D}^{#1}}				
\newcommand{\erru}[1]{\mathbf{e}_h^{#1}}			
\newcommand{\errp}[1]{\mathrm{d}_h^{#1}}			
\newcommand{\etab}{\bm{\eta}}

\newcommand{\jump}[1]{\llbracket#1\rrbracket}		
\newcommand{\Ex}{\pazocal{E}}						
\newcommand{\restr}[2]{{%
  \left.\kern-\nulldelimiterspace#1\right|_{#2}}}	
\newcommand{\Int}{\pazocal{I^\ast}}					

\DeclareMathOperator{\meas}{meas}					
\DeclareMathOperator{\Id}{Id}						

\DeclareMathOperator{\dist}{dist}					
\DeclareMathOperator{\trace}{tr}					
\DeclareMathOperator*{\esssup}{ess\,sup}			

\newcommand{\nf}[2]{\nicefrac{#1}{#2}}

\newcommand{\eps}{\varepsilon}						
\newcommand{\dt}{\Delta t}							
\newcommand{\st}{\;\vert\;}							
\newcommand{\leqc}{\lesssim}	 					

\definecolor{darkred}{RGB}{139,0,0}
\definecolor{mediumblue}{RGB}{0,0,205}
\definecolor{forestgreen}{RGB}{34,139,34}

\hypersetup{%
  linkcolor=darkred,%
  urlcolor=forestgreen,%
  citecolor=mediumblue%
}

\begin{document}

\title{\Large An unfitted Eulerian finite element method for the time-dependent
Stokes problem on moving domains}
\author{\scshape\normalsize Henry von Wahl\thanks{Corresponding author:
henry.vonwahl@ovgu.de} \textsuperscript{,}\thanks{Institut für Analysis und
Numerik, Otto-von-Guericke Univeristät, Universitätsplatz 2, {D-39106}
Magdeburg}
\and
\scshape\normalsize Thomas Richter\footnotemark[2]
\and
\scshape\normalsize Christoph Lehrenfeld\thanks{Institut für Numerische und
Angewandte Mathematik, Georg-August Universität Göttingen, Lotzestra\ss{}e
16-18, {D-37083} Göttingen}
}

\date{\small\today}
\maketitle

\begin{abstract}
We analyse a Eulerian Finite Element method, combining a Eulerian
time-stepping scheme applied to the time-dependent Stokes equations using the 
CutFEM approach with inf-sup stable Taylor-Hood elements for the spatial 
discretisation. This is based on the method introduced by Lehrenfeld \&
Olshanskii [ESAIM: M2AN 53(2):585--614] in the context of a scalar convection-
diffusion problems on moving domains, and extended to the non-stationary 
Stokes problem on moving domains by Burman, Frei \& Massing [arXiv:1910.03054
[math.NA]] using stabilised equal-order elements. The analysis includes the
geometrical error made by integrating over approximated levelset domains in the
discrete CutFEM setting. The method is implemented and the theoretical results
are illustrated using numerical examples.
\end{abstract}



\section{Introduction}\label{sec:introduction}

Flow problems on time-dependent domains are important in many
different applications in biology,  physics and engineering, such as
blood-flows~\cite{AmbrosiQuarteroniRozza2012} or fluid-structure
interaction problems~\cite{Richter2017}, e.g., freely moving solids
submersed in a fluid or generally in multi-phase flow applications,
e.g. coupling a liquid and gas. 

The most well established methods for these kind of problems studied in the
literature are either based on an at least partially Lagragian or a purely
Eulerian description of the domain boundary and its motion. Famous are
Arbitrary Lagragian-Eulerian (ALE) methods \cite{Donea1982,doneahuerta04}
which can be 
combined with rather standard time stepping schemes or space--time Galerkin
formulations \cite{behr01,klaij06,behr08,neumueller13}.  ALE methods rely on
geometry--fitted moving meshes or space--time meshes. The motion of
corresponding meshes and their necessary adaptations after significant
geometry deformations can be an severe burden for those methods, depending on
the amount of geometrical change.  To circumvent the problem of regular mesh
updates or space--time meshing, Eulerian methods can be considered. This is
also what we will do in this paper. Here, a static computational background
mesh is used to define potential unknowns and the geometry is incorporated
seperately resulting. In the context of finite element methods (FEM) these
\emph{unfitted} finite element methods have become popular within the last
decade and are known under different names, e.g. XFEM
\cite{fries2010extended}, CutFEM \cite{BCH+14}, Finite Cell Method
\cite{parvizianduesterrank07}, TraceFEM \cite{olshanskii2009finite}. Similar
concepts have also been used before, e.g. in penalty methods
\cite{babuska73b,barrettelliott86}, the fictitious domain method
\cite{glowinskietal94,glowinskietal94b}, the immersed boundary method
\cite{peskinmcqueen89}.  While these methods reached a considerable level of
maturity for stationary problems, problems with moving geometries -- one of
the main targets for unfitted finite element methods -- are not established
that much. This is due to the fact that standard time stepping approaches are
not straight-forwardly applicable in an Eulerian framework, where the
expression $\partial_t \u\approx \Delta t^{-1}(\u^n-\u^{n-1})$ is not well
defined if $\u^n$ and $\u^{n-1}$ live on different domains. 

One approach that has been proven to work despite this inconvenient setting is
the class of space--time Galerkin formulations in an Eulerian setting as they
have been considered in \cite{LR_SINUM_2013,L_SISC_2015,Z16,Pre18} for scalar
bulk problems, on moving surfaces in
\cite{GOReccomas,olshanskii2014eulerian,olshanskii2014error} or for
surface-bulk coupled problems in \cite{HLZ16}. Recently, also preliminary
steps towards unfitted space--time finite elements for two-phase flows have
been addressed in \cite{voulis18,2019arXiv191207426A}. These methods however
have the disadvantage that one has to deal -- in one way or another -- with a
higher dimensional problem. Using an adjusted quadrature rule to reduce
the space--time formulation to a classical time stepping scheme as 	
in~\cite{FreiRichter2017} calls for the costly computation of  
projections between discrete function spaces at times $t_{n-1}$ and $t_n$.

In the following we consider an alternative to space--time methods which
allows for a more standard time stepping structure and has been introduced in
\cite{LOX_ARXIV_2017,LO19} and considered for flow problems in \cite{BFM19}.
To discretise the time-derivative in the spatially smooth setting, the method
applies a standard method-of-lines approach using a backward--differentiation
formula in combination with a continuous extension operator in Sobolev spaces.
This extension ensures that the previous solution defined on the domain
$\O(t_{n-1})$ has meaning on the domain $\O(t_n)$.
In the fully discrete setting the approach uses an unfitted finite element
method on a fixed background mesh to discretise the domain and applies
additional stabilisation outside the physical domain, such that the discrete
solution has meaning on a larger domain $\OO(\O_h(t_{n-1})) \supset \O_h(t_n)$.
The major challenge with this approach, in the context of fluid problems, is
that the velocities at different time points, are weakly divergence-free with
respect to different pressure spaces. This means that the approximated
time-derivative $\dt^{-1}(\u_h^n-\u_h^{n-1})$ is not weakly divergence-free,
which causes stability problems for the pressure, as is known from the fitted
case \cite{BW11,Brenner_2013}.

The essential idea, to extend the discrete fluid velocity solution onto the
active mesh at the next time-step using ghost-penalties, in order to use a
method-of-lines approach for the discretisation of the time-derivative, has
also been considered in \cite[Section 3.6.3]{Sch17}. However, there the
extension of the velocity solution has been split into a separate sub-step and
has been limited to the extension to a vertex patch of previously active
elements, such that the time-step must obey a CFL-condition $\dt\leq c
h$. Furthermore, the analysis of this split approach currently remains open.

In the recent work \cite{BFM19}, the implicit extension technique introduced
in \cite{LO19} was also considered for the time-dependent Stokes problem on
moving domains. Here, the spatial discretisation consists of equal-order
pressure stabilised unfitted finite elements. However, the analysis in this
work was restricted to the situation, where the physical domain coincides the
the discrete levelset domain.

The remainder of this paper is structured as follows. In 
Section~\ref{sec:problem-description} we formally describe the smooth problem
we aim to solve numerically. We then begin with the temporal
semi-discretisation of the smooth problem in
Section~\ref{sec:time-discretisation}, where we also show the stability of the
approach in the spatially smooth setting.
Section~\ref{sec:fully-discrete-method} then covers the description of the
fully discrete problem, including the stabilisation operator which realises the
discrete extension in the method. The main part of this paper is then the fully
discrete analysis of the method in Section~\ref{sec:analysis}. Here wo go into
the details of how the geometrical consistency error, inherent in unfitted
finite element methods using levelsets, affect the coupling of the velocity and
pressure errors. We then present numerical examples for the method in
Section~\ref{sec:numerical-examples}. Here we show the dominance of the
geometrical error in practice, as well as discussing approaches to avoid
this issue. Finally, in Section~\ref{sec:conclusion-open-problems} we then
discuss the conclusions from this work, and discuss remaining open problems
connected to this method.


\section{Problem Description}\label{sec:problem-description}

Let us consider a time-dependent domain $\O(t)\subset\R^d$ with $d\in\{2,3\}$
and Lipschitz continuous boundary $\G(t)=\partial\O(t)$ over a fixed and
bounded time-interval $[0,T]$ and assume that this domain evolves smoothly
in time. We then define the space-time domain as  $\Q\coloneqq
\bigcup_{t\in(0,T)}\O(t)\times\{t\}$. In $\Q$ we the consider the
time-dependent Stokes problem: Find the velocity $\u$ and pressure $p$, such
that
\begin{subequations}\label{eqn:StrongProblem}
	\begin{align}
		\partial_t\u - \nu\Delta\u + \nabla p &= \f\\
		\nabla\cdot\u &= 0
	\end{align}
\end{subequations}
together with homogeneous Dirichlet boundary conditions on the space-boundary,
initial condition $\u(0) = \u_0$ and a forcing term $\f(t)$ with
the viscosity $\nu>0$.

For the well-posedness of this problem, we refer to \cite[Section~2.1]{BFM19}.


\section{Time discretisation}
\label{sec:time-discretisation}

For simplicity, let us consider a uniform time-step $\dt=T/N$ for some fixed
$N\in\N$. We then denote $t_n=n\dt$, $I_n=[t_{n-1},t_n)$, $\On=\O(t_n)$ and
$\G^n=\G(t_n)$. We define the $\delta$-neighbourhood of $\O(t)$ as
\begin{equation*}
	\Od{\O(t)}\coloneqq \{ \x\in\R^d\st \dist(\x,\O(t))\leq\delta \}.
\end{equation*}
For our method, we require that the domain $\On$ lies within the
$\delta$-neighbourhood of the previous domain, i.e.,
\begin{equation*}
	\On \subset \Od{\O^{n-1}},\quad\text{for }n=1,\dots,N.
\end{equation*}
We ensure that this requirement is fulfilled by the choice
\begin{equation*}
	\delta = c_\delta w^n_\infty\dt
\end{equation*}
where $w^n_\infty$ is the maximal normal speed of the domain interface and
$c_\delta>1$.

The time discretisation is then based on a \emph{method of lines} approach, in
combination with an extension operator for Sobolev functions to a
$\delta$-neighbourhood, which ensures that solutions defined on domains at
previous time-steps are well defined on $\On$.  See Section
\ref{sec:time-disc::subsubsec:extention-operator} for details of this
extension operator.

\subsection{Variational formulation}
\label{sec:time-disc::subsec:variational-formulation}
\subsubsection{Notation}
\label{sec:time-disc::subsubsec:notation}
We introduce some notation. By $\Ltwo{S}$ we denote the function space of
square integrable functions on a domain $S$ while $\pazocal{H}^{1}(S)$ is the
usual Sobolev space of functions in $\Ltwo{S}$ which have first order weak
derivatives in $\Ltwo{S}$. We define the subspace of $\Ltwo{S}$ of functions
with mean value zero $\Ltwozero{S} := \{ v \in \Ltwo{S} \mid \int_S v \dif\x =
0\}$ and the subspace of $\pazocal{H}^{1}(S)$ of functions with zero boundary
values (in the trace sense) as $\pazocal{H}^{1}_0(S)$. The dual space to
$(\pazocal{H}^{1}_0(S), \Vert \nabla \cdot \Vert_{S})$ is denoted by
$\pazocal{H}^{-1}(S)$. For vector-valued functions we write those spaces
bold. Further, we introduce the Poincar\'e constant $c_P > 0$ which ensures
that for all $v \in \pazocal{H}^{1}_0(\Omega(t))$ and all $t \in (0,T)$ there
holds $\Vert v \Vert_{\Omega(t)} \leq c_P \Vert \nabla v \Vert_{\Omega(t)}$.

\subsubsection{Semi-discretization}
\label{sec:time-disc::subsubsec:semi-discretization}
We discretise the time derivative with the implicit Euler (or BDF1) method in
combination with the extension operator.
Multiplying \eqref{eqn:StrongProblem} with a test function, integrating over
$\On$ and using integration by parts to obtain the weak formulations for the
diffusion and velocity-pressure coupling terms, the variational formulation of
the temporally semi-discrete problem then reads: For $n=1,\dots,N$, given
$\u^{n-1} \in \H{1}{0}{\O^{n-1}}$ and $\f^n \in \H{-1}{}{\Omega^n}$, find
$(\u^n,p^n)\in\V^n\times Q^n := \H{1}{}{\On}\times\Ltwozero{\On}$ such that
for all $(\v,q)\in\V^n\times Q^n$ it holds
\begin{equation}\label{eqn:VariationalFormulationSpaceCont}
	\frac{1}{\dt}\lprod{\u^n}{\v}{\On} + a^n(\u^n,\v) + b^n(p^n,\v) +
b^n(q,\u^n) = \dual{\f^n}{\v}{(\V^{n})',\V^n} +
\frac{1}{\dt}\lprod{\Ex\u^{n-1}}{\v}{\On}.
\end{equation}
Here, $(\cdot,\cdot)_{\On}$ denotes the $\pazocalbf{L}^2$-inner product and the
bilinear forms are
\begin{equation*}
	a^n(\u,\v) = \nu\int_{\On}\nabla\u:\nabla\v\dif\x\quad\text{and}\quad
b^n(q,\v) = -\int_{\On}q\nabla\cdot\v\dif\x
\end{equation*}
for the diffusion term and the velocity-pressure coupling respectively.
$\Ex:\H{1}{}{\O^{n-1}}\rightarrow\H{1}{ }{\Od{\O^{n-1}}}$ is the extension
operator -- discussed in the next section -- that allows us to make sense of
the ``initial value'' $\u^{n-1} \in \H{1}{0}{\O^{n-1}}$ in $\On \subset
\OO_\delta(\O^{n-1})$.

\subsubsection{Extension operator}
\label{sec:time-disc::subsubsec:extention-operator}

For the extension operator, we require the following family of space-time
anisotropic spaces
\begin{equation*}
	\def\arraystretch{1.2}
	\L{\infty}{0,T;\H{k}{}{\O(t)}}\coloneqq \left\{\v\in\L{2}{\Q}\;\middle\vert\;
		\begin{array}{l}
			\v(\cdot,t)\in\H{k}{}{\O(t)}\text{ for a.e. }t\in(0,T)\\
			\qquad\text{ and } \esssup_{t\in(0,T)}
				\nrm{\v(\cdot,t)}{\H{k}{}{\O(t)}}<\infty
		\end{array} \right\}
\end{equation*}
for $k=0,\dots,m+1$. We then denote $\partial_t\v=\v_t$ as the weak partial
derivative with respect to the time variable, if this exists as an element in
the space-time space $\L{2}{\Q}$. We also use the standard notation for
$\pazocal{L}^2$-norms of $\nrm{\cdot}{S}=\nrm{\cdot}{\pazocal{L}^2(S)}$ for
some domain $S$.

We now assume the existence of a spatial extension operator
\begin{equation*}
	\Ex:\L{2}{\O(t)} \rightarrow \L{2}{\OO(\O(t))}
\end{equation*}
which fulfils the following properties
\begin{assumption}\label{assumption:ExtensionProperties}	
Let $\v\in\L{\infty}{0,T;\H{m+1}{}{\O(t)}}\cap\W{2,\infty}{\Q}$. There exist
positive constants $c_{\ref{assumption:ExtensionProperties}a}$,
$c_{\ref{assumption:ExtensionProperties}b}$ and
$c_{\ref{assumption:ExtensionProperties}c}$ that are uniform in $t$, such that
\begin{subequations}
	\begin{align}
		\nrm{\Ex\v}{\H{k}{}{\Od{\O(t)}}} &\leq
			c_{\ref{assumption:ExtensionProperties}a}\nrm{\v}{\H{k}{}{\O(t)}}
			\label{eqn:ExtentionExtimate}\\
		\nrm{\nabla(\Ex\v)}{\Od{\O(t)}} &\leq
			c_{\ref{assumption:ExtensionProperties}b}\nrm{\nabla\v}{\O(t)}
			\label{eqn:GradExtentionExtimate}\\
		\nrm{\Ex\v}{\W{2,\infty}{\Od{\Q}}} &\leq
			c_{\ref{assumption:ExtensionProperties}c}\nrm{\v}{\W{2,\infty}{\Q}}
			\label{eqn:W2ExtentionExtimate}
	\end{align}
\end{subequations}
holds. Furthermore, if for $\v\in\L{\infty}{0,T;\H{m+1}{}{\O(t)}}$ it holds
for the weak partial time-derivative that
$\v_t\in\L{\infty}{0,T;\H{m}{}{\O(t)}}$, then
\begin{equation}\label{eqn:TimeDerivExtentionExtimate}
	\nrm{(\Ex\v)_t}{\H{m}{}{\Od{\O(t)}}} \leq
	c_{\ref{assumption:ExtensionProperties}d}\left[\nrm{\v}{\H{m+1}{}{\O(t)}} +
	\nrm{\v_t}{\H{m}{}{\O(t)}}\right],
\end{equation}
where the constant $c_{\ref{assumption:ExtensionProperties}d}>0$ again only
depends on the motion of the spatial domain. 
\end{assumption}

\begin{remark}
Such an extension operator can be constructed explicitly from the classical
linear and continuous \emph{universal} extension operator for Sobolev spaces
(see
e.g.~\cite[Section VI.3]{Ste70}) when the motion of the domain can be described
by a diffeomorphism $\Psi(t)\colon\O_0\rightarrow\O(t)$ for each $t\in[0,T]$
from the reference domain $\O_0$ that is smooth in time. See \cite{LO19} for
details thereof. 
\end{remark}

Assuming sufficient regularity of the domain $\On$, well-posedness of
\eqref{eqn:VariationalFormulationSpaceCont} is given for every time-step by the
standard theory of the Stokes-Brinkman problem. 

\subsection{Stability}
\label{sec:time-disc::subsec:stability}

We now show that the semi-discrete scheme
\eqref{eqn:VariationalFormulationSpaceCont} gives a stable solution for both
the velocity and pressure.

\begin{lemma}\label{lemma:StabilitySpaceContVel}
Let $\{\u^n\}_{n=1}^N$ be the velocity solution of
\eqref{eqn:VariationalFormulationSpaceCont} with initial data $\u^0\in\H{1}{
}{\O^0}$. Then it holds that
\begin{equation*}
	\nrm{\u^k}{\O{k}}^2 + \dt\sum_{n=1}^k\frac{\nu}{2}\nrm{\nabla\u^n}{\On}^2 
	\leq \exp(c_{L\ref{lemma:StabilitySpaceContVel}}t_k)\left[
	\nrm{\u^0}{\O^0}^2 +\frac{\nu\dt}{2}\nrm{\nabla\u^0}{\O^0}^2
	+\frac{c_P^2 \dt }{\nu}\sum_{n=1}^k\nrm{\f^n}{\On}^2\right]
\end{equation*}
with a constant $c_{L\ref{lemma:StabilitySpaceContVel}}$ independent of the
time step $\dt$ and the number of steps $k$.
\end{lemma}

\begin{proof}
The proof is analog to that of \cite[Lemma~3.6]{LO19} with the different choice
of test-function $2\dt(\u^n,-p^n) \in \V^n \times Q^n$ to remove the pressure
from the equation.

\end{proof}

\begin{lemma}\label{lemma:StabilitySpaceContPre}
Let $\{p^n\}_{n=1}^N$ be the pressure solution of
\eqref{eqn:VariationalFormulationSpaceCont}. Then it holds that
\begin{equation}\label{eqn:SharperPressureEst}
	\nrm{p^n}{\On} \leq \frac{1}{\beta} \left( c_P\nrm{\f^n}{\On} +
	\frac{1}{\dt}\nrm{\u^n-\Ex\u^{n-1}}{\H{-1}{}{\On}} +
        \nu\nrm{\nabla\u^n}{\On} \right).
\end{equation}
with a constant $\beta>0$ bounded independent of $\dt$.
\end{lemma}

\begin{proof}
Let $\beta^n$ be the inf-sup constant of the space pair $\V^n\times Q^n=
\H{1}{0}{\On}\times\Ltwozero{\On}$ and denote $\beta =
\min_{n=0,\dots,N}\beta^n$. From the inf-sup stability of the velocity and
pressure spaces and by rewriting the momentum balance equation, it follows
that
\begin{equation*}
	\beta\nrm{p^n}{\On} \leq \beta^n\nrm{p^n}{\On} \leq \sup_{\v\in\V^n}
	\frac{b^n(p^n,\v)}{\nrm{\nabla\v}{\On}} =
	\sup_{\v\in\V^n}\frac{\lprod{\f^n}{\v}{\On} -
	\frac{1}{\dt}\lprod{\u^n-\Ex\u^{n-1}}{\v}{\On} - a^n(\u^n,\v)
	}{\nrm{\nabla\v}{\On}}.
\end{equation*}
Taking absolute values, applying Cauchy-Schwarz, using the Poincaré inequality
and the continuity of the diffusion bilinear form we get the claim.
\end{proof}

\begin{remark}\label{rem:nodt}
At first glance it seems unclear if the estimate in
\eqref{eqn:SharperPressureEst} yields a pressure bound that is independent of
$\dt$. Let us explain why a scaling of $\Vert p^n \Vert_{\Omega^n}$ with
$\dt^{-1}$ is not to be expected. The argument is based on an a relation for
the discretisation error.
The exact solution $(\u(t_n),p(t_n))$ fulfills for all $\v \in \V^n$ and all $q
\in Q^n$ an equation similar to \eqref{eqn:VariationalFormulationSpaceCont}:
\begin{equation*} 
	\llprod{\partial_t \u(t_n)}{\v}{\On} + a^n(\u(t_n),\v) + b^n(p(t_n),\v) +
	b^n(q,\u(t_n)) = \dual{\f^n}{\v}{(\V^{n})',\V^n}.
      \end{equation*}
We find for $\Eu{k} := \u(t_k) - \u^k$, and $\Ep{k} := p(t_k) - p^k$,
$k=1,..,N$ that there holds
\begin{align*}
& \llprod{ \frac{\Eu{n} - \Ex  \Eu{n-1}}{\dt}}{\v}{\O^n} + a^n(\Eu{n},\v) +
b^n(\Ep{n},\v) + b^n(q,\Eu{n})  = \llprod{
\frac{\u(t_{n}) - \Ex  \u(t_{n-1})}{\dt} -\partial_t \u(t_{n})}{\v}{\O^n}
\end{align*}
for all $\v \in \V^n$ and $q \in Q^n$. 
Now, assuming sufficient regularity, i.e. $u \in \W{2,\infty}{\Q}$,
we easily obtain the bound
$c_{R\ref{rem:nodt}a} \dt \nrm{\u}{\W{2,\infty}{\Q}} \nrm{\v}{\O^n}$
for the r.h.s. with a constant $c_{R\ref{rem:nodt}a}$ independent of $n$, $\u$
and $\dt$. Here, we also
made use of \eqref{eqn:W2ExtentionExtimate}.
As the l.h.s. is the same as in \eqref{eqn:VariationalFormulationSpaceCont} we
can apply Lemma \ref{lemma:StabilitySpaceContPre} (using $\Eu{0}=0$) to obtain
the bound
\begin{equation*}
\nrm{\Eu{n}}{\O^n} \leq c_{R\ref{rem:nodt}b} \exp(c_{R\ref{rem:nodt}c}t_n) \dt
\nrm{\u}{\W{2,\infty}{\Q}}
\end{equation*}
Hence, a simple triangle inequality yields
\begin{align*}
	\frac{1}{\dt}\nrm{\u^n-\Ex\u^{n-1}}{\H{-1}{}{\On}} 
		&\leq\! \frac{1}{\dt}\nrm{\u(t_n)-\Ex\u(t_{n-1})}{\H{-1}{}{\On}}
\!+\! \frac{1}{\dt} \nrm{\Eu{n}}{\H{-1}{}{\On}} + \frac{1}{\dt} \nrm{ \Ex
\Eu{n-1}}{\H{-1}{}{\On}} \\
& \leq \nrm{\partial_t \u(t_n)}{\H{-1}{}{\On}} + c_{R\ref{rem:nodt}d} \dt
\Vert \u \Vert_{\W{2,\infty}{\Q}}.
\end{align*}
and hence a bound on the norm of $p^n$ that is independent on $\dt^{-1}$.
\end{remark}


\section{The fully discrete method}
\label{sec:fully-discrete-method}

For the spatial discretisation of the method we use the CutFEM approach
\cite{BCH+14}. Within the CutFEM framework, we consider a background domain
$\widetilde\O\subset\R^d$ and assume that $\O(t)\subset\widetilde\O$ for all
$t\in[0,T]$. We then take a simplicial, shape-regular and quasi-uniform mesh
$\Tht$ of $\Ot$, where $h>0$ is the characteristic size of the simplices.
Bad cuts of the mesh with the domain boundary $\G(t)$ are stabilised using
ghost-penalty stabilisation \cite{Bur10}. We will discuss the details of this
in Section~\ref{sec:fully-disc-meth::subsubsec:ghost-penalties}. The fully
discrete method realises the necessary extension implicitly, by applying
the ghost-penalty stabilisation operation on a larger  extension region when
solving for the solution in each time step. At each time step, we therefore
extend the (discrete) physical domain $\Onh$ by a strip of width
\begin{equation*}
	\delta_h=c_{\delta_h} w^n_\infty\dt
\end{equation*} 
such that $\O^{n+1}_h$ is a subset of this extended domain with
$c_{\delta_h}>1$, but sufficiently small so that $\OO_{\delta_h}(\Onh) \subset
\OO_{\delta}(\On)$.  We collect all elements which have some part in this
extended domain at time $t=t_n$ in the \emph{active velocity mesh}, denoted as
\begin{equation*}
	\T^n_{h,\delta_h} \coloneqq \{ K\in\Tht\st \exists\x\in K 
	\text{ such that }\dist(\x,\Onh)\leq\delta_h \} \subset \Tht
\end{equation*}
and denote the \emph{active domain} as
\begin{equation*}
	\OdT{n} \coloneqq \{ \x\in K \st K\in\T_{h,\delta_h} \} \subset\R^d .
\end{equation*}
We further define the \emph{cut mesh} as $\Th^n = \T^n_{h,0}$ of all elements
which contain some part of the physical domain and the \emph{cut domain} as
$\OO^n_{\T}=\OO^n_{0,\T}$.

On the active mesh, we consider an inf-sup stable finite element pair
$\V_h\times Q_h$ for the Nitsche-CutFEM discretisation of the Stokes problem
\cite{BH14,MLLR14}. For an overview of such elements see \cite{GO17}. We shall
use the family of Taylor-Hood finite elements for $k\geq 2$
\begin{equation*}
	\Vnh \coloneqq \{\v_h\in \bm{C}(\OdT{n} ) \st
	\restr{\v_h}{K}\in[\P{k}{K}]^d\text{ for all } K\in\Thn \}
\end{equation*}
and
\begin{equation*}
	\Qnh \coloneqq \{ q_h\in C(\OO^n_{\T} ) \st \restr{\v_h}{K}\in\P{k-1}{K}
	\text{ for all } K\in\T_h^n \}.
\end{equation*}


\subsection{The variational formulation}
\label{sec:fully-disc-meth::subsec:variational-formulation}

The fully discrete variational formulation of the method reads as follows:
Given an appropriate initial condition $\u^0_h\in\V_h^0$, for $n=1,\dots,N$
find $(\u_h^n,p_h^n)\in\Vnh\times\Qnh$, such that
\begin{equation}\label{eqn:VariationalFormulationDiscrete}
	\int_{\O^n_h}\frac{\u_h^n-\u_h^{n-1}}{\Delta t}\cdot \v_h\dif\x +
	a^n_h(\u_h^n,\v_h) + b_h^n(p_h^n,\v_h) + b_h^n(q_h,\u_h^n) +
	s_h^n((\u_h^n,p_h^n),(\v_h,q_h)) = \f^n_h(\v_h)
\end{equation}
for all $(\v_h,q_h)\in\V_h^n\times Q_h^n$. We impose the homogeneous Dirichlet
boundary conditions in a weak sense using the symmetric Nitsches method
\cite{Nit71}. For the diffusion term, the bilinear form is then
\begin{align*}
	a^n_h(\u_h,\v_h) &\coloneqq \nu\int_{\Onh}\nabla\u:\nabla\v\dif\x + \nu
		N^n_h(\u_h,\v_h),\\
	N^n_h(\u_h,\v_h) &\coloneqq N^n_{h,c}(\u_h,\v_h)+N^n_{h,c}(\v,\u) +
		N^n_{h,s}(\u_h,\v_h)
\end{align*}
where
\begin{equation*}
	N^n_{h,c}(\u_h,\v_h)\coloneqq\int_{\G^n_h}(-\partial_{\n} \u_h)\cdot\v_h\dif
	s\quad\text{and}\quad N^n_{h,s}(\u_h,\v_h)\coloneqq\int_{\G^n_h}
	\frac{\sigma}{h}\u_h\cdot\v_h\dif s
\end{equation*} 
are the consistency (symmetry) and penalty terms of Nitsche's method while
$\sigma>0$ is the penalty parameter. The velocity-pressure coupling term is
given by
\begin{equation*}
	b^n_h(\v,q) = -\int_{\Onh} q\nabla\cdot\v\dif\x 
		+ \int_{\G^n_h} p\v\cdot\n\dif s.
\end{equation*}

In order to realise the discrete extension of the velocity and to stabilise
the method with respect to essentially arbitrary mesh-interface cut positions,
we use ghost-penalty stabilisation. This term is given by
\begin{equation*}
	s_h^n((\u_h^n,p_h^n),(\v_h,q_h)) = \gamma_{s,u}\nu i^n_h(\u_h^n,\v_h)
	 +  \gamma_{s,u}'\frac{1}{\nu} i^n_h(\u_h^n,\v_h)
	 -\gamma_{s,p}\frac{1}{\nu}j_h^n(p_h^n,q_h)
\end{equation*}
with stabilisation parameters $\gamma_{s,u},\gamma_{s,u}', \gamma_{s,p} >
0$. A suitable choice for these parameters will be discussed later, cf. Remark
\ref{rem:choiceofgp} below. The velocity ghost-penalty operator
$i_h^n(\cdot,\cdot)$ stabilises the velocity w.r.t. arbitrary bad cut
configurations and implicitly defines an extension of the velocity field. It
will therefore act on a strip of elements both inside and outside the physical
domain, in order for us to have control over the velocity on the entire active
domain $\Thn$.  The pressure ghost-penalty operator stabilises the pressure in
the $\pazocal{L}^2$-norm and is needed to give an inf-sup property for
unfitted finite elements \cite{GO17}. This operator will therefore only act in
the direct vicinity of the domain boundary and only on elements which have at
least some part in the physical domain.


\subsubsection{The ghost-penalty operator}
\label{sec:fully-disc-meth::subsubsec:ghost-penalties}

The stabilisation bilinear form $s_h(\cdot,\cdot)$ has two purposes here.
First, it stabilises the discrete problem
\eqref{eqn:VariationalFormulationDiscrete} with respect to domain
boundary-mesh cut position and it implicitly provides the extension of the
velocity field that is needed to allow the method of lines approach. To the
best of our knowledge, there are currently three different versions of the
ghost-penalty stabilisation operator. An \emph{LPS-type} version was the first
ghost-penalty operator, proposed in \cite{Bur10}. The \emph{normal derivative
jump} version is probably the most widely used variant, see e.g.,
\cite{BCH+14,BH12,MLLR14,SW14,GM19}. We will use the \emph{direct version} of
the ghost penalty operator introduced in \cite{Pre18}. For details on all
three ghost penalty operators, see \cite{LO19}. We will only provide details
on the direct version used here. However, the other versions of the
ghost-penalty operator could also be used instead.

For the velocity ghost-penalty operator, we define the set of elements in the
boundary strip
\begin{equation*}
	\TSpm \coloneqq \{ K\in\Tht\st \exists\x\in K\text{ with}
	\dist(\x,\G^n_h)\leq\delta_h\}
\end{equation*}
and the set of interior facets in this strip
\begin{equation*}
	\Fhnd \coloneqq \{ F=\overline{T}_1\cap \overline{T}_2\st T_1\in\Thn
	,T_2\in\TSpm ,T_1\neq T_2\text{ and } \meas_{d-1}(F)>0 \}.
\end{equation*}
Furthermore, for the pressure ghost-penalty operator we define the set of
\emph{boundary elements}
\begin{equation*}
	\ThGhn \coloneqq \{K\in\Th^n\st\exists\x\in K \text{ with }\x\in\Ghn \},
\end{equation*}
and the set of interior facets of these elements
\begin{equation*}
	\Fhn \coloneqq \{ F=\overline{T}_1\cap \overline{T}_2\st T_1\in\T_h^n
	,T_2\in\ThGhn ,T_1\neq T_2\text{ and } \meas_{d-1}(F)>0 \}.
\end{equation*}
A sketch of the different sets of element and facets for both the velocity and
pressure can be seen in Figure~\ref{fig:mesh-sketch}.

\begin{figure}
	\centering
	\def\myfilename{velocity-mesh-elements}	
	\subfloat[][Active velocity elements]{
	\includegraphics[scale=4, valign=t]{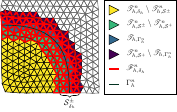}
	}
	\hskip.02\textwidth
	\subfloat[][Active pressure elements]{
	\vphantom{\includegraphics[scale=4, valign=t]{\myfilename}}
	\def\myfilename{pressure-mesh-elements}	
	\includegraphics[scale=4, valign=t]{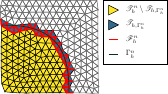}
	\undef\myfilename
	}
	\caption{The different stability and extension elements for both the velocity
		and pressure. Note that the interior (yellow) elements which have a red
		facet are also used by the direct ghost penalty operator.}
	\label{fig:mesh-sketch}
\end{figure}

To define the stabilisation operator, we require some further notation. For a
facet in the extention strip $\Fhnd\ni F =\overline{T}_1\cap \overline{T}_2$
let $\omega_F=T_1\cup T_2$ be the corresponding \emph{facet patch}. We then
consider $\jump{u}\coloneqq u_1-u_2$ with $u_i=\restr{\E u}{T_i}$, where
$\E\colon \P{m}{K}\rightarrow\P{m}{\R^d}$ is the canonical extension of a
polynomial to $\R^d$.

The velocity ghost-penalty operator is then defined as
\begin{equation*}
	i^n_h(\u_h^n,\v_h) = \sum_{F\in\Fhnd}\frac{1}{h^2}
	\int_{\omega_F}\jump{\u_h}\cdot\jump{\v_h}\dif\x
\end{equation*}
and the pressure ghost-penalty operator is defined as
\begin{equation*}
	j_h^n((p_h^n,q_h)) = 
	\sum_{F\in\Fhn}\int_{\omega_F}\jump{p}\cdot\jump{q}\dif\x.
\end{equation*}

For the analysis, we will also need to insert general $\pazocal{L}^2$-functions
as arguments of the ghost-penalty operators. In this case we take
$u_i=\restr{\E\Pi_{T_i}u}{T_i}$, where $\Pi_{T_i}$ is the
$\pazocal{L}^2(T_i)$-projection onto $\P{m}{T_i}$.

\begin{lemma}[Consistency]\label{lemma:GP-consitency}
Let $\w\in\H{m+1}{ }{\OdT{n}}$ and $r\in\pazocal{H}^{m}(\OO^n_{\T})$,
$n=1,\dots,N$, $m \geq 1$. Then it holds that
\begin{align*}
	i_h^n(\w,\w) &\leq c h^{2m}\nrm{\w}{\H{m+1}{ }{\OdT{n}}}^2\\
	j_h^n(r,r) &\leq c h^{2m}\nrm{r}{\pazocal{H}^{m}(\OO^n_{\T})}^2.
\end{align*}
Furthermore, let $\Int$ be the Lagrangian interpolation operator for the
velocity ($\Int: [C^0(\OdT{n})]^d \to \Vnh$) and for the pressure space
($\Int: [C^0(\OdT{n})]^d \to \Qnh$). Then we also
have
\begin{align*}
	i_h^n(\w-\Int\w,\w-\Int\w) &\leq
		ch^{2m}\nrm{\w}{\H{m+1}{}{\OdT{n}}}^2\\
	j_h^n(r-\Int r,r-\Int r) &\leq 
		ch^{2m}\nrm{r}{\pazocal{H}^{m}(\OdT{n})}^2.
\end{align*}
\end{lemma}

\begin{proof} We note that for $d\leq 3$ and $m \geq 1$ $\pazocalbf{H}^{m+1}$
is compactly embedded into $C^0$ so that $\Int$ is well-defined. For the
estimates with $i_h^n(\cdot,\cdot)$ we refer to \cite[Lemma~5.8]{LO19}. The
proof for $j_h^n(\cdot,\cdot)$ follows analogously with the different $h$
scaling in the definition of $j_h^n(\cdot,\cdot)$.
\end{proof}

\section{Analysis of the method}\label{sec:analysis}

The analysis of the fully discrete method in this section is structured as
follows. In Section~\ref{sec:analysis::subsec:preliminaries} we will introduce
further notation, concepts and basic necessary results from the literature
needed for the analysis. Section~\ref{sec:analysis::subsec:unique-solvability}
will then cover the existence and uniqueness of the solution to the fully
discretised system. The stability of this solution is then discussed in 
Section~\ref{sec:analysis::subsec:stability}. We then cover some technical
details on the geometry approximation made by integrating over discrete
approximations $\Onh$ of the exact domain $\On$ in
Section~\ref{sec:analysis::subsec:geometrical-approximation}. With the tools
covered in these sections, we then show the consistency of the method in both
time and space in Section~\ref{sec:analysis::subsec:consitency}. With this
consistency result we are then able to prove an error estimate for the solution
in the energy norm in Section~\ref{sec:analysis::subsec:error-estimates}.

\subsection{Preliminaries}\label{sec:analysis::subsec:preliminaries}

For the analysis we require some further notation and definitions. We define
the \emph{extension strip mesh} as
\begin{equation*}
	\TSp \coloneqq \{K\in\Tht\st \exists\x\in\Ot\setminus\Onh\text{ with }
	\dist(\x,\Ghn)\leq \delta_h \}.
\end{equation*}
We also define the \emph{sharp strips} as
\begin{equation*}
	\Spm_{\delta_h}(\Onh) \coloneqq \{\x\in\Ot\st\dist(\x,\Ghn)\leq\delta_h\}
	\quad\text{and}\quad \Sp_{\delta_h}(\Onh)\coloneqq
	\{\x\in\Ot\setminus\Onh\st\dist(\x,\Ghn)\leq\delta_h \}.
\end{equation*}
Furthermore, we define the \emph{discrete extended domain} $\Odh{\Onh}\coloneqq
\Spm_{\delta_h}(\Onh)\cup\Onh$. In the analysis we require that $\delta$ is
sufficiently large, such that
\begin{equation}\label{eqn:DiscreteDomainInclusions}
	\OdT{n}\subset\Od{\On}\quad\text{and}\quad\Onh\subset\Od{\O(t)},\,t\in I_n =
	[t_{n-1}, t_n),
\end{equation}
for $n=1,\dots,N$. Furthermore, for ease and brevity of notation we write
$a\leqc b$ if it holds that $a\leq cb$ with a constant $c>0$ independent of the
mesh size $h$, the time step $\dt$, the time $t$ and the mesh-interface cut
position. Similarly, we write $a\simeq b$ if both $a\leqc b$ and $b \leqc a$
holds.

For the analysis we shall consider the following mesh dependent norms. For the
velocity we take
\begin{align*}
	\tripnrm{\v}_n^2 &\coloneqq\nrm{\nabla\v}{\Onh}^2 +
		\nrm{h^{-\nf{1}{2}}\v}{\Ghn}^2 + 
		\nrm{h^{\nf{1}{2}}\partial_{\n}\v}{\Ghn}^2\\
	\tripnrmast{\v}{n}^2 &\coloneqq \nrm{\nabla\v}{\OdT{n}}^2 +
		\nrm{h^{-\nf{1}{2}}\v}{\G^n_h}^2, \\
	\nrm{\v}{-1,n} &\coloneqq
		\sup_{\w\in\Vnh}\frac{\lprod{\v}{\w}{\Onh}}{\tripnrmast{\w}{n}}
\end{align*}
for the pressure
\begin{align*}
	\tripnrm{q}^2_n &\coloneqq \nrm{q}{\Onh}^2 + \nrm{h^{\nf{1}{2}}q}{\Ghn}^2 \\
	\tripnrmast{q}{n} &\coloneqq \nrm{q}{\OO^n_{\T}}
\end{align*}
and for the product space
\begin{align*}
	\tripnrmast{(\u,p)}{n}^2 &\coloneqq  \tripnrmast{\u}{n}^2 +
	\tripnrmast{p}{n}^2.
\end{align*}
Note that $\tripnrm{\cdot}_{n}$-norms are defined on the physical domain and
add control on the normal derivative of the velocity and the trace of the
pressure at the boundary. This norm arises naturally to bound the bilinear
form $a_h^n(\u,\v)$ for functions $\u,\v \in \H{1}{}{\Onh}$. The second type
of norms, the $\tripnrmast{\cdot}{n}$-norms, are defined on the entire active
domain and therefore represent proper norms for discrete finite element
functions.

\subsubsection{Basic estimates}
\label{sec:analysis::subsec:prelim::subsubsec:estimates}

For $v_h \in \P{k}{K},~K\in \Tht$ we have the inverse and trace estimates:
\begin{subequations}
	\begin{align}
          \nrm{\nabla v_h}{K} &\leqc h_K^{-1}\nrm{v_h}{K},
\label{eqn:InvEst:DerivEl}\\
          \nrm{h^{\nf{1}{2}}\partial_{\n}v_h}{F} &\leqc \nrm{\nabla v_h}{K},
\label{eqn:TraceEstNormDerivFacet}\\
\nrm{h^{\nf{1}{2}}\partial_{\n}v_h}{K\cap\Ghn} &\leqc \nrm{\nabla v_h}{K}
\label{eqn:TraceEstNormDerivBnd}
	\end{align}
\end{subequations}
      
For \eqref{eqn:InvEst:DerivEl} and \eqref{eqn:TraceEstNormDerivFacet} see for
example \cite{Qua14}. For a proof of \eqref{eqn:TraceEstNormDerivBnd}, see
\cite{HH02}. Furthermore, for $v\in\pazocal{H}^{1}(K),~K \in \Thn$ we have the
following trace inequality
	\begin{align}
		\nrm{v}{K\cap\Ghn} &\leqc h^{-\nf{1}{2}}_K\nrm{v}{K} +
			h^{\nf{1}{2}}_K\nrm{\nabla v}{K} \label{eqn:TraceIneqBnd}
	\end{align}
See \cite{HH02} for \eqref{eqn:TraceIneqBnd}. It follows from these estimates
that
\begin{equation}\label{eqn:TripNormEst}
	\tripnrm{\v_h}_n \leqc \tripnrmast{\v_h}{n}\quad \text{and}
	\quad\tripnrm{q_h}_n\leqc \tripnrmast{q_h}{n}
\end{equation}
for all $\v_h\in\V_h^n$ and $q_h\in Q_h^n$. Furthermore, we have a discrete
version of the Poincaré inequality
\begin{equation}\label{eqn:Poincare:tripnrm}
	\nrm{\v_h}{\OdT{n}} \leq c_{P,h} \tripnrmast{\v_h}{n}
\end{equation}
for all $\v_h\in\V_h$, see \cite[Lemma 7.2]{MLLR14} for a proof thereof.

\subsubsection{Interpolation properties}
\label{sec:analysis::subsubsec:interpolation}

\begin{lemma}\label{lemma:InterpolationEstimate}

Let $K\in\Tht$ and $w\in\pazocal{H}^{m+1}(K)$, $m\geq 1$. Then it holds for the
Lagrange interpolant
\begin{equation*}
	\nrm{D^k(w-\Int w)}{K} \leq h^{m+1-k}\nrm{D^{m+1}w}{K},
	\qquad 0\leq k\leq m,
\end{equation*}
and on the boundary of an element
\begin{equation*}
	\nrm{w-\Int w}{\partial K} \leq h^{m+\nf{1}{2}}\nrm{D^{m+1}w}{K}.
\end{equation*}
\end{lemma}

\begin{proof}
	See for example \cite{Ape99}.
\end{proof}

\begin{lemma}\label{lemma:InterpolationEstTripNrm}
Let $\v\in\H{m+1}{}{\On}$ and $q\in\pazocal{H}^m(\On)$, $m\geq 1$. Then it
holds that
\begin{subequations}
\begin{align}
	\tripnrm{\Ex \v-\Int \Ex \v}_n &\leqc h^m\nrm{\v}{\H{m+1}{}{\On}}
		\label{eqn:TripnrmInterpolVel} \\
	\tripnrm{\Ex q-\Int \Ex q}_n &\leqc h^m\nrm{q}{\pazocal{H}^m(\On)}.
		\label{eqn:TripnrmInterpolPre}
\end{align}
\end{subequations}
\end{lemma}
\begin{proof}
See also \cite[Lemma~4.1]{MLLR14}. For \eqref{eqn:TripnrmInterpolVel}, we
consider the volume and boundary contribution to the $\tripnrm{\cdot}_n$-norm
separately. Let $\w = \Ex \v$.
	
We split the volume terms into element contributions. On interior elements, we
simply apply Lemma~\ref{lemma:InterpolationEstimate}. For cut simplices, we
extend the norm onto the entire element before applying the interpolation
estimate. Summing up over all active elements and applying
\eqref{eqn:ExtentionExtimate} in combination with
\eqref{eqn:DiscreteDomainInclusions} gives
\begin{equation*}
	\nrm{\nabla(\w-\Int\w)}{\Onh} \leqc h^m\nrm{\w}{\H{m+1}{}{\OdT{n}}} \leqc
	h^m\nrm{\v}{\H{m}{}{\On}}.
\end{equation*}
Similarly, on each cut element we can estimate the boundary terms using
\eqref{eqn:TraceIneqBnd} and apply the interpolation estimate:
\begin{equation*}
	\nrm{h^{\nf{1}{2}}\partial_{\n}(\w-\Int\w) }{K\cap\Ghn}\leqc
	\nrm{\nabla(\w-\Int\w)}{K} + h\nrm{\nabla^2(\w-\Int\w)}{K}\leqc
	h^m\nrm{\w}{\H{m+1}{}{K}}
\end{equation*}
and
\begin{equation*}
	\nrm{h^{-\nf{1}{2}}(\w-\Int\w)}{K\cap\Ghn}\leqc h^{-1}\nrm{\w-\Int\w}{K} +
	\nrm{\nabla(\w-\Int\w)}{K} \leqc h^m\nrm{\w}{\H{m+1}{}{K}}.
\end{equation*}
Summing up over all cut elements and again applying
\eqref{eqn:ExtentionExtimate} in combination with
\eqref{eqn:DiscreteDomainInclusions} then proves
\eqref{eqn:TripnrmInterpolVel}. The proof for \eqref{eqn:TripnrmInterpolPre}
follows along the same lines of argument and in the case of $m=1$ we consider
the Clément interpolation operator.
\end{proof}

\subsubsection{The ghost-penalty mechanism}
\label{sec:analysis::subsubsec:ghost-penalty-mechanism}

\begin{assumption}\label{assumption:StripWidth}
We assume that for every strip element $K\in\TSp$ there exists an uncut
element $K'\in\Thn\setminus\TSp$ which can be reached by a path which crosses a
bounded number of facets $F\in\Fhn$. We assume that the number of facets which
have to be crossed to reach $K'$ from $K$ is bounded by a constant $L\leqc
(1+\frac{\delta_h}{h})$ and that every uncut element $K'\in\Thn\setminus\TSp$
is the end of at most $M$ such paths with $M$ bounded independent of $\dt$ and
$h$. In other words, each uncut elements "supports" at most $M$ strip elements.
\end{assumption}

See \cite[Remark~5.4]{LO19} for a justification as to why the above assumption
is reasonable if the mesh resolves the domain boundary sufficiently well.

\begin{lemma}\label{lemma:TriNrmEquiv}

  For all $\v_h\in\V_h^n$ and $q_h\in Q_h^n$ it holds that
\begin{subequations}
\begin{align}
	\nrm{\nabla\v_h}{\OdT{n}}^2 &\simeq \nrm{\nabla\v_h}{\Onh}^2 + L\cdot
		i^n_h(\v_h,\v_h) \label{eqn:VelNormEqiv}\\
	\nrm{\v_h}{\OdT{n}}^2 &\simeq \nrm{\v_h}{\Onh}^2 + h^2 L\cdot i^n_h(\v_h,\v_h)
		\label{eqn:VelL2NormEqiv}\\
	\nrm{q_h}{\OO^n_{\T}}^2 &\simeq \nrm{q_h}{\Onh}^2 +
		j^n_h(q_h,q_h).\label{eqn:PreNormEqiv}
\end{align}
\end{subequations}
\end{lemma}

\begin{proof}
For the first inequality in both \eqref{eqn:VelNormEqiv} and
\eqref{eqn:PreNormEqiv} we refer to \cite[Lemma~5.5]{LO19}. For the second
bound we use the fact that we can bound the direct version of the ghost penalty
operator by the normal derivative jump version, see \cite[Chapter~3,
Remark~6]{Pre18}. The upper bound then follows by an application of the
inverse (trace) estimates \eqref{eqn:InvEst:DerivEl} and
\eqref{eqn:TraceEstNormDerivFacet}, see \cite[Proposition~5.1]{MLLR14}.
\end{proof}

\subsection{Unique solvability}
\label{sec:analysis::subsec:unique-solvability}

\begin{lemma}[Continuity]\label{lemma:continuity}
For the diffusion bilinear form we have for all $\v,\w\in\H{1}{}{\Od{\On}}$
that it holds
\begin{equation}\label{eqn:DiffusionContinuity}
	a_h^n(\v,\w) \leqc \nu\tripnrm{\v}_n\tripnrm{\w}_n
\end{equation}
and for all $\v_h,\w_h\in\V_h^n$ it holds that
\begin{equation}\label{eqn:DiffusionGPContinuity}
	a_h^n(\v_h,\w_h) + \nu L i_h^n(\v_h,\w_h)
	\leqc \nu \tripnrmast{\v_h}{n}\tripnrmast{\w_h}{n}.
\end{equation}
Furthermore, for the velocity-pressure coupling bilinear form we have for all
$q\in\pazocal{L}^2(\Od{\On})$ and $\v\in\H{1}{}{\Od{\On}}$ that
\begin{equation*}
	b_h^n(q,\v) \leqc \tripnrm{q}_n\tripnrm{\v}_n
\end{equation*}
and for all $q_h\in Q_h^n$ and $\v_h\in\V_h^n$ that
\begin{equation*}
	b_h^n(q_h,\v_h) \leqc \tripnrmast{q_h}{n}\tripnrmast{\v_h}{n}.
\end{equation*}
\end{lemma}

\begin{proof}
See \cite{BH12} for the diffusion bilinear form and \cite{MLLR14} for the
pressure coupling bilinear form. The scaling of the ghost-penalty term in
\eqref{eqn:DiffusionGPContinuity} is due to the larger extension strip and 
Lemma~\ref{lemma:TriNrmEquiv}.
\end{proof}

\begin{lemma}[Coercivity]\label{lemma:coercivity}
There exists a constant $c_{L\ref{lemma:coercivity}}>0$, independent of $h$
and the mesh-interface cut position, such that for sufficiently large $\sigma >
0$ there holds
\begin{equation*}
	a^n_h(\u_h,\u_h) + \nu L i^n_h(\u_h,\v_h) \geq \nu c_{L\ref{lemma:coercivity}}
	\tripnrmast{\u_h}{n}^2
\end{equation*}
for all $\u_h\in\V_h$.
\end{lemma}

\begin{proof}
See for example \cite[Lemma~6]{BH12} or \cite[Lemma~4.2]{BH14}. The
different scaling is again due to Lemma~\ref{lemma:TriNrmEquiv}.
\end{proof}

\begin{lemma}\label{lemma:bad-inf-sup}

Let $\On_{h,i}\coloneqq \Onh\setminus \{\x\in K\st K\in \ThGhn
\}$ denote the interior, uncut domain. It then holds for all $q_h\in\Qnh$ with
$\restr{q_h}{\On_{h,i}}\in\Ltwozero{\On_{h,i}}$ that
\begin{equation}\label{eqn:bad-inf-sup}
	\beta\nrm{q_h}{\Onh} \leq \sup_{\v_h\in\Vnh}
	\frac{b_h(q_h,\v_h)}{\tripnrmast{\v_h}{n}} + j_h^n(q_h,q_h)^{\nf{1}{2}}.
\end{equation}
The constant $\beta>0$ is independent of $h$ and $q_h$.
\end{lemma}

\begin{proof}
	See \cite[Corollary 1]{GO17} for the proof.
\end{proof}

\begin{remark}[Choice of ghost-penalty parameter] \label{rem:choiceofgp} From
Lemma~\ref{lemma:continuity} and Lemma~\ref{lemma:coercivity} we see that the
velocity ghost-penalty parameter should scale with the strip-width $L$.  This
is necessary, in order for an exterior unphysical but active element to obtain
support from an uncut interior element for which we have to cross at most $L$
elements to reach it, c.f. Assumption~\ref{assumption:StripWidth}. As this
first part of the ghost-penalties is related to the stabilization of the
viscosity bilinear form $a_h^n(\cdot,\cdot)$ only it has a scaling with
$\nu$. We require the same mechanism also for the implicit extension of
functions. As we will see in the analysis below this requires the same
ghost-penalty, however with the different scaling $\nf{1}{\nu}$.

The pressure ghost-penalty operator in Lemma~\ref{lemma:bad-inf-sup} does not
need a scaling with $L$ as we require these ghost-penalties only for
stabilizing the velocity-pressure coupling, but not for the extension of the
pressure field into a $\delta$-neighbourhood.

For simplicity of the analysis, we choose a common ghost-penalty stabilisation
parameter $\gamma_s$, which we use to set $\gamma_{s,u} = \gamma_{s,u}' =
L\gamma_s$ and $\gamma_{s,p} = \gamma_s$.
\end{remark}

We now collect all fully implicit non-ghost-penalty terms in the bilinear form

\begin{equation*}
	A_h^n((\u_h^n,p_h^n),(\v_h,q_h)) \coloneqq
	\frac{1}{\dt}\lprod{\u_h^n}{\v_h}{\Onh} + a^n_h(\u_h^n,\v_h) + 
	b_h^n(p_h^n,\v_h)+ b_h^n(q_h,\u_h^n)
\end{equation*}
and the explicit terms in the linear form
\begin{equation*}
	F_h^n(\v_h) \coloneqq \frac{1}{\dt}\lprod{\u_h^{n-1}}{\v_h}{\Onh} + 
	\f^n_h(\v_h).
\end{equation*}
We can then rewrite problem \eqref{eqn:VariationalFormulationDiscrete} as:
Find $(\u_h^n,p_h^n)\in\Vnh\times\Qnh$, such that
\begin{equation}\label{eqn:MatrixProblem}
	A_h^n((\u_h^n,p_h^n),(\v_h,q_h)) + s_h^n((\u_h^n,p_h^n),(\v_h,q_h)) =
	F_h^n(\v_h^n)
\end{equation}
for all $(\v_h,q_h)\in\V_h^n\times Q_h^n$. This problem is indeed well-posed:

\begin{theorem}[Well posedness]\label{theorem:DiscreteWellPosedness}
Consider the norm $\tripnrmdt{(\u_h^n,p_h^n)}{n}^2 \coloneqq
\frac{1}{\dt}\nrm{\u_h^n}{\Onh}^2 + \tripnrmast{(\u_h^n,p_h^n)}{n}^2$. There
exists a constant $c_{T\ref{theorem:DiscreteWellPosedness}}>0$ such that
for all $(\u_h^n,p_h^n)\in\Vnh\times\Qnh$ it holds that
\begin{equation}\label{eqn:BigMatrix-inf-sup-result}
	\sup_{(\v_h,q_h)\in\Vnh\times\Qnh} \frac{A_h^n((\u_h^n,p_h^n),(\v_h,q_h)) +
	s_h^n((\u_h^n,p_h^n),(\v_h,q_h))}{\tripnrmdt{(\v_h,q_h)}{n}} \geq
	c_{T\ref{theorem:DiscreteWellPosedness}}\tripnrmdt{(\u_h^n,p_h^n)}{n}
\end{equation}
The solution $(\u_h^n,p_h^n)\in\Vnh\times\Qnh$ to \eqref{eqn:MatrixProblem}
exists and is unique.
\end{theorem}

\begin{proof}
The proof follows the same lines as for the unfitted Stokes problem using 
CutFEM in \cite{MLLR14,GO17}. Let $(\u^n_h,p^n_h)\in\Vnh\times\Qnh$ be given 
and let $\w_h\in\Vnh$ be the test function, such that \eqref{eqn:bad-inf-sup}
holds and w.l.o.g that $\tripnrmast{\w_h}{n}=\nrm{p^n_h}{\Onh}$. Testing
\eqref{eqn:MatrixProblem} with $(-\w_h,0)$ and using Cauchy-Schwarz,
Lemma~\ref{lemma:continuity}, Lemma~\ref{lemma:bad-inf-sup} and the weighted
Young's inequality then gives
\begin{align*}
	(A_h^n+s_h^n)((\u_h^n,p_h^n),(-\w_h,0)) 
	&= -\frac{1}{\dt}\lprod{\u_h^n}{\w_h}{\Onh}
		- \big(a_h^n+(\nu+\nf{1}{\nu})\big)(\u_h^n,\w_h)
		+ b_h^n(-p_h^n,\w_h)\\
	&\geq 
		\begin{multlined}[t]
			- \frac{1}{\dt}\nrm{\u^n_h}{\Onh}\nrm{\w_h}{\Onh}
			- c\tripnrmast{\u^n_h}{n}\tripnrmast{\w_h}{n}
			+ \beta\nrm{p^n_h}{\Onh}\tripnrmast{\w_h}{n}\\
			- j_h^n(p^n_h,p^n_h)^{\nf{1}{2}}
			\tripnrmast{\w_h}{n}
		\end{multlined}\\
	&\geq 
		\begin{multlined}[t]
			-\frac{1}{2\eps_1\dt^2}\nrm{\u^n_h}{\Onh}^2
			- \frac{c}{2\eps_2}\tripnrmast{\u_h^n}{n}^2
			- \frac{1}{2\eps_3}j_h^n(p_h^n,p_h^n)\\
			+ (\beta-\frac{\eps_1}{2}-\frac{c_1 \eps_2}{2}-\frac{\eps_3}{2})
				\nrm{p_h^n}{\Onh}^2
		\end{multlined}
\end{align*}
Choosing $\eps_1,\eps_2,\eps_3>0$ such that $(\beta - \nf{\eps_1}{2} -
\nf{c\eps_2}{2} - \nf{\eps_3}{2})>0$ we then have
\begin{equation*}
	(A_h^n+s_h^n)((\u_h^n,p_h^n),(-\w_h,0)) \geq
		-\frac{c_1}{\dt^2}\nrm{\u_h^n}{\Onh}
		- c_2\tripnrmast{\u^n_h}{n}^2 
		+ c_3 \nrm{p^n_h}{\Onh}^2 - c_4 j_h^n(p^n_h,p^n_h).
\end{equation*}
Furthermore, with the test function $(\v_h,q_h)=(\u^n_h,-p^n_h)$ in
\eqref{eqn:MatrixProblem} and using Lemma~\ref{lemma:coercivity} we have
\begin{equation*}
	(A_h^n+s_h^n)((\u_h^n,p_h^n),(\u^n_h,-p^n_h)) \geq
		\frac{1}{\dt}\nrm{\u_h^n}{\Onh}^2
		+ c_{L\ref{lemma:coercivity}}\tripnrmast{\u^n_h}{n}^2 
		+ j_h^n(p^n_h,p^n_h).
\end{equation*}
Combining these two estimates, we have with a suitable choice of $\dt\delta>0$
that
\begin{equation*}
	(A_h^n+s_h^n)((\u_h^n,p_h^n),(\u^n_h-\dt\delta\w_h,-p^n_h)) \geq
	c\tripnrmdt{(\u^n_h,p^n_h)}{n}^2.
\end{equation*}
Since $\tripnrmdt{(\u^n_h-\dt\delta\w_h,-p^n_h)}{n} \leq (1+\dt\delta)
\tripnrmdt{(\u^n_h,p^n_h)}{n}$, we then have
\eqref{eqn:BigMatrix-inf-sup-result}.

The continuity of $(A^n_h+s_h^n)$ in the $\tripnrmdt{(\cdot,\cdot)}{n}$-norm
can be easily shown. The existence and uniqueness of the solution then follows
by the Banach-Nec\vs as-Babu\vs ska Theorem (see e.g.
\cite[Theorem~2.6]{EG04}).
\end{proof}

\subsection{Stability}\label{sec:analysis::subsec:stability}

For the fully discrete method, we have a discrete counterpart of
Lemma~\ref{lemma:StabilitySpaceContVel}.

\begin{theorem}\label{theorem:StabilityDiscrete}
For the velocity solution $\u_h^n\in\Vnh$, $n=1,\dots,N$ of
\eqref{eqn:VariationalFormulationDiscrete} we have the stability estimate
\begin{multline*}
	\nrm{\u_h^k}{\O_n^k}^2 + \dt \sum_{n=1}^k  \Big[ \frac{\nu
		c_{T\ref{theorem:StabilityDiscrete}a}}{2}
		\tripnrmast{\u_h^n}{n}^2 + \frac{L}{\nu}i_h^n(\u_h^n,\u_h^n) \Big] \\
	\leq \exp(c_{T\ref{theorem:StabilityDiscrete}b}\nu^{-1}t_k)\left[
		\nrm{\u_h^0}{\O_n^0}^2 + \frac{\nu\dt
		c_{L\ref{theorem:StabilityDiscrete}a}}{2}\tripnrmast{\u_h^0}{0}^2
		+ \frac{L}{\nu}i_h^0(\u_h^0,\u_h^0)
		+ \dt \sum_{n=1}^k\frac{c_{P,h}^2}{\nu 
			c_{T\ref{theorem:StabilityDiscrete}a}} \nrm{\f_h^n}{\On_h}^2
		\right].
\end{multline*}	
\end{theorem}

\begin{proof}
Testing \eqref{eqn:VariationalFormulationDiscrete} with $(\v_h,q_h) =
2\dt(\u_h^n,-p_h^n)\in\Vnh\times\Qnh$ and using the identity
\begin{equation}\label{eqn:TimeDiffIdentity}
	2\lprod{\u^n-\Ex\u^{n-1}}{\u^n}{\On} = \nrm{\u^n}{\On}^2 +
	\nrm{\u^n-\Ex\u^{n-1}}{\On}^2 - \nrm{\Ex\u^{n-1}}{\On}^2
\end{equation}
gives us
\begin{multline*}
	\nrm{\u_h^n}{\Onh}^2 + \nrm{\u_h^n-\u_h^{n-1}}{\Onh}^2 + 2\dt[
		a_h^n(\u_h^n,\u_h^n) + \nu\gamma_s i_h^n(\u_h^n,\u_h^n)
		+\frac{\gamma_s}{\nu}i_h^n(\u_h^n,\v_h^n)
		+\frac{\gamma_s}{\nu}j_h^n(p_h^n,p_h^n) ]\\
	= 2\dt\f_h^n(\u_h^n) + \nrm{\u_h^{n-1}}{\Onh}^2.
\end{multline*}
Using the coercivity result from Lemma~\ref{lemma:coercivity} and the fact that
$\nrm{\u_h^n-\u_h^{n-1}}{\Onh}^2 + \frac{2\dt\gamma_s}{\nu} j_h^n(p_h^n,p_h^n)
\geq 0$, we get
\begin{equation}\label{eqn:StabilityProof1}
	\nrm{\u_h^n}{\Onh}^2 
		+ 2\dt\nu c_{L\ref{lemma:coercivity}}\tripnrmast{\u_h^n}{n}^2
		+ \dt \frac{2}{\nu}L i_h^n(\u_h^n,\u_h^n)
	 \leq 2\dt\f_h^n(\u_h^n) + \nrm{\u_h^{n-1}}{\Onh}^2
\end{equation}
To bound the forcing term we use the Cauchy-Schwarz inequality, the discrete 
Poincaré-inequality \eqref{eqn:Poincare:tripnrm} as well as the weighted
Young's inequality which gives
\begin{equation*}
	\f_h^n(\u_h^n)\leq\nrm{\f_h^n}{\On_h}\nrm{\u_h^n}{\On_h} \leq	
	\nrm{\f_h^n}{\On_h}c_{P,h}\tripnrmast{\u_h^n}{n}\leq
	\frac{c_{P,h}^2\nrm{\f_h^n}{\On_h}^2}{2\eps}
	+ \frac{\eps\tripnrmast{\u_h^n}{n}^2}{2}.
\end{equation*}
Inserting this into \eqref{eqn:StabilityProof1} with the choice $\eps=\nu
c_{L\ref{lemma:coercivity}}$ then gives
\begin{equation}\label{eqn:StabilityProof2}
	\nrm{\u_h^n}{\Onh}^2 
		+ \dt\nu c_{L\ref{lemma:coercivity}}\tripnrmast{\u_h^n}{n}^2  
		+ \dt \frac{2}{\nu}L i_h^n(\u_h^n,\u_h^n)
	\leq \frac{c_{P,h}^2\dt}{\nu c_{L\ref{lemma:coercivity}}}
		\nrm{\f_h^n}{\On_h}^2
		+ \nrm{\u_h^{n-1}}{\Onh}^2
\end{equation}
To obtain a bound on $\nrm{\u_h^{n-1}}{\Onh}^2$ we utilise the following
result, c.f., \cite[Lemma~5.7]{LO19}:
\begin{equation}\label{eqn:estimate-on-old-domain}
	\nrm{\u_h}{\Odh{\On_h}}^2 \leq (1+c_1(\eps)\dt)\nrm{\u_h}{\On_h}^2 +
	c_2(\eps)\nu\dt\nrm{\nabla\u_h}{\On_h}^2 + c_3(\eps,h)\dt L i_h^n(\u_h,\u_h)
\end{equation}
with $c_1(\eps)=c'c_{\delta_h}w^n_\infty(1+\eps^{-1})$,
$c_2(\eps)=c'c_{\delta_h}w^n_\infty\eps/\nu$ and
$c_3(\eps)=c'c_{\delta_h}w^n_\infty(\eps + (1+\eps^{-1})h^2)$ and $c'>0$
independent of $h$ and $\dt$. Choosing $\eps=\nu c_{L\ref{lemma:coercivity}}/
(2c'c_{\delta_h}w^n_\infty)$ we then have $c_2=c_{L\ref{lemma:coercivity}}/2$
while we can bound the resulting $c_1(\eps)\leq \overline c/nu$, with
$\overline c$ independent of $\dt$ and $h$ and $c_3 \leq \nu
c_{L\ref{lemma:coercivity}} /2 + h^2 \overline c/\nu$. This gives the estimate
\begin{align*}
	\nrm{\u_h^{n-1}}{\Onh}^2 
		&\leq \nrm{\u_h^{n-1}}{\Odh{\O_h^{n-1}}}^2 \\
		&\leq (1 + \frac{\overline{c}\dt}{\nu})\nrm{\u_h^{n-1}}{\O_h^{n-1}}^2
		+ \frac{\nu\dt c_{L\ref{lemma:coercivity}}}{2}\tripnrmast{\u^{n-1}}{n-1}^2
		+ \frac{\overline c h^2}{\nu}\dt L i_h^{n-1}(\u_h^{n-1},\u_h^{n-1})
\end{align*}
Inserting this into \eqref{eqn:StabilityProof2} and summing over $n=1,\dots,k$
for $k\leq N$ then gives for sufficiently small $h$, such that $\overline c h^2
\leq 1$
\begin{multline*}
	\nrm{\u_h^k}{\O_n^k}^2 
		+ \dt\sum_{n=1}^k\Big[\frac{\nu c_{L\ref{lemma:coercivity}}}{2}
			\tripnrmast{\u_h^n}{n}^2 + \frac{L}{\nu}i_h^n(\u_h^n,\u_h^n) \Big]\\ 
	\leq \nrm{\u_h^0}{\O_n^0}^2 
		+ \frac{\nu\dt c_{L\ref{lemma:coercivity}}}{2}\tripnrmast{\u_h^0}{0}^2
		+ \frac{L}{\nu}i_h^0(\u_h^0,\u_h^0)
		+ \dt \sum_{n=1}^k \frac{c_{P,h}^2}{\nu c_{L\ref{lemma:coercivity}}}
			\nrm{\f_h^n}{\On_h}^2 
		+ \dt\sum_{n=0}^k \frac{\overline c}{\nu}\nrm{\u_h^n}{\Onh}^2.
\end{multline*}
Applying a discrete Gronwall inequality with $
c_{T\ref{theorem:StabilityDiscrete}a} = c_{L\ref{lemma:coercivity}}$ and $
c_{T\ref{theorem:StabilityDiscrete}b} = \overline c$ then gives the desired
result.
\end{proof}

\begin{lemma}\label{lemma:discrete-pressure:stability}

For the pressure solution $p_h^n\in\Qnh$ of
\eqref{eqn:VariationalFormulationDiscrete} it holds that
\begin{equation}\label{eqn:bad-pressure-stability}
	\nrm{p_h^n}{\Onh} \leq c_{L\ref{lemma:discrete-pressure:stability}}
	\left[ \nrm{\frac{1}{\dt}(\u_h^n - \u_h^{n-1})}{-1,n} +
	\tripnrmast{\u_h^n}{n} + \nrm{\f_h^n}{\Onh} + j_h^n(p_h^n,p_h^n)^{\nf{1}{2}}
	\right].
\end{equation}

\end{lemma}

\begin{proof}
We have that
\begin{align*}
	b_h(p_h^n,\v_h) &= -\frac{1}{\dt}\lprod{\u_h^n -
		\u_h^{n-1}}{\v_h}{\Onh} - (a_h^n+\nu L i_h^n)(\u_h^n,\v_h) 
		- \nf{1}{\nu}i_h^n(\u_h^n,\v_h) + \f_h^n(\v_h)\\
	&\leq \left[\nrm{\frac{1}{\dt}(\u_h^n - \u_h^{n-1})}{-1,n} +
		c(\nu + \nf{1}{\nu})\tripnrmast{\u_h^n}{n}
		+ c_{P,h}\nrm{\f_h^n}{\Onh} \right] \tripnrmast{\v_h}{n}
\end{align*}
Here we used the continuity of $a_h^n + \nu L i_h^n$, the estimate
$i_h^n(\w,\v) \leq i_h^n(\w,\w)^{\nf{1}{2}}i_h^n(\v,\v)^{\nf{1}{2}}$ with
\eqref{eqn:VelNormEqiv} and the Poincaré inequality. The result then follows
from the inf-sup result in Lemma~\ref{lemma:bad-inf-sup}.
\end{proof}

\begin{remark}
Lemma~\ref{lemma:discrete-pressure:stability} does not immediately result 
in a satisfactoery $\pazocal{L}^2$-stability estimate for the pressure. 
\eqref{eqn:bad-pressure-stability} only gives us an estimate for
\begin{equation*}
	\dt^2\sum_{n=1}^k\nrm{p_h^n}{\Onh}^2.
\end{equation*}
In order to get a proper estimate of the form $\dt \sum_{n=1}^k
\nrm{p_h^n}{\Onh}^2$ we require an estimate of the term $\frac{1}{\dt}
\nrm{\u_h^n-\u_h^{n-1}}{-1,n}$ which is independent (of negative powers) of
$\dt$.
\end{remark}

\subsection{Geometrical Approximation}
\label{sec:analysis::subsec:geometrical-approximation}

We shall assume that we have a higher order approximation of the geometry,
i.e.,
\begin{equation*}
	\dist(\On,\Onh) \leq h^{q+1}
\end{equation*}
with the geometry approximation order $q$ and that integrals on $\Onh$ can be
computed sufficiently accurately. We further assume the existence of a mapping
which maps the approximated extended domain to the exact extended domain. In
other words, we have $\Phi:\Odh{\Onh}\rightarrow\Odh{\On}$ which we assume to
be well-defined, continuous, it holds that $\On=\Phi(\Onh)$,
$\G^n=\Phi(\G^n_h)$ and $\Odh{\On}=\Phi(\Odh{\Onh})$ and
\begin{equation}\label{eqn:MappingEstimates}
	\begin{gathered}
		\nrm{\Phi - \Id }{\Linf{\Odh{\Onh}}} \leqc h^{q+1},\\
		\nrm{D\Phi - I}{\Linf{\Odh{\Onh}}}\leqc h^q,\quad
		\nrm{\det(D\Phi)-1}{\Linf{\Odh{\Onh}}}\leqc h^q.
	\end{gathered}
\end{equation}
Furthermore, for sufficiently small $h$, the mapping $\Phi$ is invertible. Such
a mapping has been constructed, e.g., in \cite[Section~7.1]{GOR15}. This
mapping $\Phi$ is used here (as in \cite{LO19}, and therefore using the same
notation) to map from the discrete domain to the exact one. Let $\v_h\in\Vnh$
and define $\vhl=\v_h\circ\Phi^{-1}$. From the third estimate in
\eqref{eqn:MappingEstimates} we have $\det(D\phi)\simeq 1$, hence we then get
using integration by substitution
\begin{equation*}
	\nrm{\vhl}{\Odh{\On}}^2 = \sum_{i=1}^{d}\int_{\Odh{\On}}(\vhl)_i^2\dif\hat\x
	= \sum_{i=1}^{d}\int_{\Odh{\Onh}}\det(D\Phi)(\v_h)_i^2\dif\x\simeq
	\nrm{\v_h}{\Odh{\Onh}}^2.
\end{equation*}
Using the same argument we also have that
\begin{equation*}
	\nrm{\v_h}{\Onh}^2 \simeq  \nrm{\vhl}{\On}^2
\end{equation*}
as well as 
\begin{equation*}
	\nrm{\nabla\v_h}{\Odh{\Onh}}^2 \simeq \nrm{\nabla\vhl}{\Odh{\On}}^2\qquad
	\text{and}\qquad \nrm{\nabla\v_h}{\Onh}^2\simeq  \nrm{\nabla\vhl}{\On}^2.
\end{equation*}
For the extension, we also have the following result, c.f.
\cite[Lemma~7.3]{GOR15}:
\begin{lemma}\label{lemma:MappingEstimates}
The following estimates
\begin{subequations}
	\begin{align}
		\nrm{\Ex\u - \u\circ\Phi}{\Onh} &\leqc h^{q+1}
			\nrm{\u}{\H{1}{}{\On}},\label{eqn:MappingEstimaLemma:L2Volume}\\
		\nrm{\nabla(\Ex\u) - (\nabla\u)\circ\Phi}{\Onh} &\leqc
			h^{q+1}\nrm{\u}{\pazocalbf{H}^2(\On)},
			\label{eqn:MappingEstimaLemma:GradL2Volume}\\
		\nrm{\Ex\u - \u\circ\Phi}{\Gamma_h^n} &\leqc h^{q+1}
		\nrm{\u}{\H{2}{}{\On}},\label{eqn:MappingEstimaLemma:L2Boundary}
	\end{align}
\end{subequations}
hold for all $\u\in\H{2}{}{\On}$, $n=1,\dots,N$. Furthermore, it also holds
that
\begin{equation}
	\nrm{\Ex\u - (\Ex\u)\circ\Phi}{\Odh{\Onh}} \leqc h^{q+1}
		\nrm{\u}{\H{1}{}{\On}},\label{eqn:MappingEstimaLemma:ExtendedL2Volume}.
\end{equation}
\end{lemma}
\begin{proof}
For \eqref{eqn:MappingEstimaLemma:L2Volume} -- 
\eqref{eqn:MappingEstimaLemma:L2Boundary} we refer to \cite[Lemma~7.3]{GOR15}. 
The proof of \eqref{eqn:MappingEstimaLemma:ExtendedL2Volume} follows the 
identical lines of \eqref{eqn:MappingEstimaLemma:L2Volume} but integration over
$\Spm(\Onh)$ rather that $\Spm(\Onh)\setminus\Sp(\Onh)$.

\end{proof}

\begin{lemma}\label{lemma:MappingEstimate2}
For all $\v\in\H{m+1}{}{\On}$ and $r\in\pazocal{H}^m(\On)$ we have the
estimates
\begin{subequations}
	\begin{align}
		h^{\nf{1}{2}}\nrm{\Ex\partial_{\n}\v - \partial_{\n}\v\circ\Phi}{\Ghn} 
			&\leqc h^q\nrm{v}{\H{1}{}{\On}} + h^{m+\nf{1}{2}}
			\nrm{\v}{\H{m+1}{}{\On}}\label{eqn:MappingEstLemma2:VelNormDeriv}\\
		\nrm{\Ex\v\cdot\n - (\v\cdot\n)\circ\Phi}{\Ghn} &\leqc h^{q+\nf{1}{2}}
			\nrm{\v}{\H{1}{}{\On}} + h^{m+\nf{1}{2}}\nrm{\v}{\H{m+1}{}{\On}}
			\label{eqn:MappingEstLemma2:VelNorm}\\
		h^{\nf{1}{2}}\nrm{\Ex r-r\circ\Phi}{\Ghn} &\leqc h^q
			\nrm{r}{\pazocal{H}^1(\On)} + h^m\nrm{r}{\pazocal{H}^m(\On)}
			\label{eqn:MappingEstLemma2:Pres}.
	\end{align}
\end{subequations}
\end{lemma}

\begin{proof}
To make the proof more readable, we do not write the extension operator
explicitly and identify $\v$ with its smooth extension $\Ex\v$.

Now let $\v\in\H{m+1}{}{\On}$. To prove
\eqref{eqn:MappingEstLemma2:VelNormDeriv} we
insert additive zeros and use the triangle inequality to get
\begin{multline}\label{eqn:h^1/2partial_n_u_est}
	h^{\nf{1}{2}}\nrm{ \partial_{\n}\v\circ\Phi-\partial_{\n}\v}{\Ghn} \leq
		h^{\nf{1}{2}}\left[ \nrm{ \partial_{\n}\v\circ\Phi - 
		\Int\partial_{\n}\v\circ\Phi}{\Ghn}\right. \\
	+ \left.\nrm{ \Int\partial_{\n}\v\circ\Phi-\Int\partial_{\n}\v}{\Ghn} + 
		\nrm{\Int\partial_{\n}\v-\partial_{\n}\v}{\Ghn}\right].
\end{multline}
To bound the first term on the right-hand side of
\eqref{eqn:h^1/2partial_n_u_est} we divide the norm up into the separate
contributions of cut elements, use \eqref{eqn:TraceIneqBnd}, the interpolation
estimate in Lemma~\ref{lemma:InterpolationEstimate} as well as
Assumption~\ref{assumption:ExtensionProperties} in combination with
\eqref{eqn:DiscreteDomainInclusions} in order to get
\begin{align*}
	\nrm{ \partial_{\n}\v\circ\Phi-\Int\partial_{\n}\v\circ\Phi}{\Ghn} 
		&= \nrm{ \partial_{\n}\v-\Int\partial_{\n}\v}{\G^n}\\
	&=\sum_{K\in\ThGn}
		\nrm{\partial_{\n}\v-\Int\partial_{\n}\v}{K\cap\G^n}\\
	&\leqc \sum_{K\in\ThGn} h^{-\nf{1}{2}}\nrm{\nabla\v-\Int\nabla\v}{K} 
		+ h^{\nf{1}{2}}\nrm{\nabla(\nabla\v - \Int\nabla\v ) }{K} \\
	&\leqc \sum_{K\in\ThGn} h^{-\nf{1}{2}}h^m\nrm{\v}{\H{m+1}{}{K}} + 
		h^{\nf{1}{2}}h^{m-1}\nrm{\v}{\H{m+1}{ }{K}}\\
	&\leqc h^{m+\nf{1}{2}}\nrm{\v}{\H{m+1}{}{\OdT{n}}} \leqc 
		h^{m+\nf{1}{2}}\nrm{\v}{\H{m+1}{}{\On}}.
		\numberthis\label{eqn:h^1/2partial_n_u_est-part1}
\end{align*}
The third term in \eqref{eqn:h^1/2partial_n_u_est} can be estimated completely
analogously. For the second term we make use of the fact, that the argument of
the norm is discrete, therefore permitting the use of the inverse estimate
\eqref{eqn:InvEst:DerivEl}. Then with \eqref{eqn:TraceEstNormDerivBnd} and
\eqref{eqn:DiscreteDomainInclusions} it follows that
\begin{align*}
	\nrm{ \Int\partial_{\n}\v\circ\Phi-\Int\partial_{\n}\v}{\Ghn} &=
		\sum_{K\in\ThGhn} \nrm{\Int\partial_{\n}\v\circ\Phi-
		\Int\partial_{\n}\v}{K\cap\Ghn}\\
	&\leqc \sum_{K\in\ThGhn} h^{-\nf{1}{2}}\nrm{\nabla(\Int\v\circ\Phi-
		\Int\v)}{K}\\
	&\leqc \sum_{K\in\ThGhn} h^{-\nf{3}{2}}\nrm{\Int(\v\circ\Phi-\v)}{K}\\
	&\leqc h^{-\nf{3}{2}}\nrm{\v\circ\Phi-\v}{\Odh{\Onh}}
\end{align*}
We can then apply \eqref{eqn:MappingEstimaLemma:ExtendedL2Volume} so that we
get
\begin{align}
	\nrm{ \Int\partial_{\n}\v\circ\Phi-\Int\partial_{\n}\v}{\Ghn} &\leqc
		h^{-\nf{3}{2}}\nrm{\v\circ\Phi-\v}{\Odh{\Onh}}\nonumber\\
	& \leqc h^{q-\nf{1}{2}}\nrm{\v}{\H{1}{}{\On}}.
		\numberthis\label{eqn:h^1/2partial_n_u_est-part2}
\end{align}
Combining \eqref{eqn:h^1/2partial_n_u_est-part1} and
\eqref{eqn:h^1/2partial_n_u_est-part2} then gives us the desired estimate
\begin{equation*}
	h^{\nf{1}{2}}\nrm{ \partial_{\n}\v\circ\Phi-\partial_{\n}\v}{\Ghn} \leqc
	h^{m+\nf{1}{2}}\nrm{\v}{\H{m+1}{}{\On}} + h^q\nrm{\v}{\H{1}{}{\On}}
\end{equation*}

For \eqref{eqn:MappingEstLemma2:Pres}, let $r\in\pazocal{H}^m(\On)$. We again
introduce two additive zeros as in \eqref{eqn:h^1/2partial_n_u_est} and apply
the triangle inequality
\begin{equation*}
	\nrm{r-r\circ\Phi}{\Ghn} \leq \nrm{r-\Int r}{\Ghn} + \nrm{\Int r-\Int
	(r\circ\Phi)}{\Ghn} + \nrm{\Int (r\circ\Phi)-r\circ\Phi}{\Ghn}.
\end{equation*}
The first term of which we can then be estimated using
\eqref{eqn:TraceIneqBnd}, Lemma~\ref{lemma:InterpolationEstimate} and
Assumption~\ref{assumption:ExtensionProperties}
\begin{align*}
	\nrm{r-\Int r}{\Ghn}^2 &= \sum_{K\in\ThGhn}\nrm{r-\Int r}{K\cap\Ghn}\\
	&\leqc \sum_{K\in\ThGhn}h^{-1}\nrm{r-\Int r}{K}^2 +
		h^{1}\nrm{\nabla(r-\Int r)}{K}^2\\
	&\leqc h^{-1}\nrm{r-\Int r}{\Odh{\Onh}}^2 + h^1\nrm{\nabla(r-
		\Int r)}{\Odh{\Onh}}\\
	&\leqc h^{2m-1}\nrm{r}{\pazocal{H}^m(\Odh{\Onh})}^2 +
		h^{2(m-1)+1}\nrm{r}{\pazocal{H}^m(\Odh{\Onh})}^2\\
	&\leqc h^{2m-1}\nrm{r}{\pazocal{H}^m(\On)}.
\end{align*}
The term $\nrm{\Int (r\circ\Phi)-r\circ\Phi}{\Ghn}$ can then be bound along the
same lines. For the term $ \nrm{\Int r-\Int (r\circ\Phi)}{\Ghn}$ we make use of
the fact that the argument of the norm is a discrete function and therefore
permitting the inverse estimate \eqref{eqn:InvEst:DerivEl}. Then with
\eqref{eqn:MappingEstimaLemma:ExtendedL2Volume} we get
\begin{align*}
	\nrm{\Int r - \Int r\circ\Phi}{\Ghn}^2 &= \sum_{K\in\ThGhn}\nrm{\Int
		r-\Int r\circ\Phi}{K\cap\Ghn} \\
	&\leqc \sum_{K\in\ThGhn}h^{-1}\nrm{\Int r-\Int r\circ\Phi}{K}^2 +
		h^{1}\nrm{\nabla(\Int r-\Int r\circ\Phi)}{K}^2\\
	&\leqc \sum_{K\in\ThGhn}h^{-1}\nrm{\Int r-\Int r\circ\Phi}{K}^2\\
	&\leqc \nrm{r-r\circ\Phi}{\Odh{\Onh}}^2\leqc h^{2q-1}
		\nrm{r}{\pazocal{H}^1(\On)}.
\end{align*}
This then gives the desired estimate
\begin{equation*}
	h^{\nf{1}{2}}\nrm{r-r\circ\Phi}{\Ghn} \leqc h^q\nrm{r}{\pazocal{H}^1(\On)} 
	+ h^m\nrm{r}{\pazocal{H}^m(\On)}.
\end{equation*}

The proof of \eqref{eqn:MappingEstLemma2:VelNorm} follows along the same lines
as the other two estimates.

\end{proof}

\subsection{Consistency}\label{sec:analysis::subsec:consitency}

Using integration by parts, we see that any smooth solution $(\u,p)$ to the
strong problem \eqref{eqn:StrongProblem}, fulfils
\begin{equation}\label{eqn:SmoothSolutionVariational}
	\int_{\Onh}\partial_t\u(t_n)\vhl\dif\hat\x + a_1^n(\u(t_n),\vhl) +
	b_1^n(p(t_n),\vhl) + b_1^n(\qhl,\u(t_n)) = \f^n(\vhl)
\end{equation}
for the test-functions $(\vhl,\qhl) = (\v_h\circ\Phi^{-1},q_h\circ\Phi^{-1})$
with $(\v_h,q_h)\in\V_h^n\times Q_h^n $, the mapping $\Phi$ from
Section~\ref{sec:analysis::subsec:geometrical-approximation} and the bilinear
forms
\begin{equation*}
	a_1^n(\u,\v) \coloneqq a^n(\u,\v) + \nu\int_{\G^n}(-\partial_{\n}\u)\cdot\v
	\dif s
\end{equation*}
and
\begin{equation*}
	b_1^n(p,\v) \coloneqq b^n(p,\v) + \int_{\G^n}p\v\cdot\n\dif s.
\end{equation*}
Since $\Onh \subset \Od{\O(t)}$, $t\in[t_{n-1},t_n]$, $\u(t_{n-1})=\u^{n-1}$ is
well defined on $\Onh$. For simplicity of notation we shall identify the smooth
extension $\Ex\u$ with $\u$ and denote $\Eu{n}\coloneqq\u^n - \u^n_h$ and
$\Ep{n} \coloneqq p^n - p^n_h$. Subtracting
\eqref{eqn:VariationalFormulationDiscrete} from
\eqref{eqn:SmoothSolutionVariational} and adding and subtracting appropriate
terms we obtain the error equation
\begin{align*}
	\int_{\Onh}&\frac{\Eu{n}-\Eu{n-1}}{\dt}\cdot\v_h\dif\x + a_h^n(\Eu{n},\v_h)
		+ b_h^n(\Ep{n},\v_h) + b_h^n(q_h,\Eu{n}) + 
		s_h^n((\Eu{n},\Ep{n}),(\v_h,q_h)) \\
	&= \!\begin{multlined}[t]
			\f^n(\vhl) - \f^n_h(\v_h) + \int_{\Onh}\frac{\u^n-\u^{n-1}}{\dt}
				\cdot\v_h\dif\x - \int_{\On}\partial_t\u^n\cdot\vhl\dif\hat\x + 
				a_h^n(\u^n,\v_h) - a_1^n(\u^n,\vhl)\\
			b_h^n(p^n,\v_h) - b_1^n(p^n,\vhl) + b_h^n(q_h,\u^n) - 
				b_1^n(\qhl,\u^n) + s_h^n((\u^n,p^n),(\v_h,q_h))
		\end{multlined}\\
	&= \mathfrak{T}_1 + \mathfrak{T}_2 + \mathfrak{T}_3 + \mathfrak{T}_4 + 
		\mathfrak{T}_5 + \mathfrak{T}_6\\
	&= \mathfrak{E}_c^n(\v_h,q_h).\numberthis\label{eqn:ErrorEquation}
\end{align*}
These six terms correspond to the forcing, time-derivative, diffusion,
pressure, divergence constraint and ghost-penalty contributions respectively.

\begin{lemma}[Consistency estimate]\label{lemma:consistency}
The consistency error has the bound
\begin{equation*}
	\vert  \mathfrak{E}_c^n((\v_h,q_h)) \vert \leqc \left(\dt + h^q 
		+ \frac{h^m L^{\nf{1}{2}}}{\nu}\right) 
			R_{c,1}(\u,p,\f)\tripnrmast{\v_h}{n} 
		+ (h^q + h^m)R_{c,2}(\u,p)
	\tripnrmast{q_h}{n}
\end{equation*}
with 
\begin{equation*}
	R_{c,1}(\u,p,\f)= \nrm{\u}{\W{2,\infty}{\Q}}
		+ \nrm{\f^n}{\H{1}{}{\On}} +
		\sup_{t\in[0,T]}\big(\nrm{\u(t)}{\H{m+1}{}{\O(t)}} +
		\nrm{p(t)}{\pazocal{H}^m(\O(t))}\big)
\end{equation*}
and
\begin{equation*}
	R_{c,2}(\u,p) = \sup_{t\in[0,T]}\big(\nrm{\u(t)}{\H{m+1}{}{\O(t)}} +
		\nrm{p(t)}{\pazocal{H}^m(\O(t))}\big).
\end{equation*}
\end{lemma}

\begin{proof}
We estimate each of the five components of the consistency error separately.

$\mathfrak{T}_1$: See also \cite[Lemma~7.5]{GOR14}, the preprint of
\cite{GOR15}. Beginning with the forcing contribution $\mathfrak{T}_1$ and
denoting $J=\det(D\Phi)$, we use integration by substitution to obtain
\begin{align*}
	\vert\mathfrak{T}_1\vert = \vert\f^n(\vhl) - \f^n_h(\v_h) \vert &= 
		\Big\vert\int_{\On}\f^n\cdot\vhl\dif\hat\x - \int_{\Onh}\f^n\cdot\v_h
		\dif\x \Big\vert = \Big\vert \int_{\Onh} (J\cdot\f^n\circ\Phi - 
		\f^n)\cdot\v_h\dif\x \Big\vert\\
	&\leq \Big\vert \int_{\Onh}J(\f^n\circ\Phi - \f^n)\cdot\v_h\dif\x\Big\vert 
		+ \Big\vert\int_{\Onh}(J-1)\f^n\cdot\v_h\dif\x\Big\vert \\
	&\leqc \nrm{\f^n\circ\Phi -\f^n}{\Onh}\nrm{\v_h}{\Onh} + 
		\nrm{J-1}{\infty}\nrm{\f^n}{\Onh}\nrm{\v_h}{\Onh}\\
	&\leqc h^{q+1}\nrm{\f^n}{\H{1}{ }{\On}}\nrm{\v_h}{\Onh} + 
		h^q\nrm{\f^n}{\Onh}\nrm{\v_h}{\Onh}\\
	&\leqc h^q \nrm{\f^n}{\H{1}{ }{\Odh{\Onh}}}\nrm{\v_h}{\Onh}\\
	&\leqc h^q \nrm{\f^n}{\H{1}{}{\On}}\nrm{\v_h}{\Onh}
		\numberthis\label{eqn:consistencyT1}.
\end{align*}
Here we used \eqref{eqn:MappingEstimates}, Lemma~\ref{lemma:MappingEstimates}
and Assumption~\ref{assumption:ExtensionProperties}.

$\mathfrak{T}_2$: The proof for the time-derivative component follows the proof
of \cite[Lemma~5.11]{LO19}. Using integration by parts in time, we have that
\begin{equation*}
	\frac{1}{\dt}\left[ \u(t_n) - \u(t_{n-1}) \right] = \u_t(t_n)
	-\int_{t_{n-1}}^{t_n}\frac{t-t_{n-1}}{\dt}\u_{tt}\dif t.
\end{equation*}
Therefore, we may rewrite $\mathfrak{T}_2$ as
\begin{equation}\label{eqn:RewriteConsistencyT2}
	\mathfrak{T}_2 = \int_{\Onh}\u_t(t_n)\cdot\v_h\dif\x -
	\int_{\On}\u_t(t_n)\cdot\vhl\dif\hat\x - \int_{\Onh}\int_{t_{n-1}}^{t_n} 
	\frac{t-t_{n-1}}{\dt}\u_{tt}\dif t\cdot\v_h\dif\x.
\end{equation}
For the first two terms of \eqref{eqn:RewriteConsistencyT2} we have
\begin{align*}
	\int_{\Onh} \u_t(t_n)\cdot\v_h \dif\x - \int_{\On} \u_t(t_n)\cdot\vhl
		\dif\hat\x &= \int_{\Onh} \u_t(t_n)\cdot\v_h \dif\x - \int_{\Phi(\Onh)}
		\u_t(t_n)\cdot(\v_h \circ\Phi^{-1}) \dif\hat\x\\
	&= \int_{\Onh} \u_t(t_n)\cdot\v_h \dif\x - \int_{\Onh}
		(\u_t\circ\Phi)(t_n)\cdot\v_h \det(D\Phi) \dif\x\\
	&= \int_{\Onh} \left[\u_t(t_n) - (\u_t\circ\Phi)(t_n) J\right]\cdot\v_h 
		\dif\x\eqqcolon \mathfrak{I}_2^1.
\end{align*}
We estimate this by
\begin{align*}
	\vert \mathfrak{I}_2^1 \vert &\leq \Big\vert  \int_{\Onh}\left[\u_t(t_n) -
		(\u_t\circ\Phi)(t_n)\right]J\cdot\v_h\dif\x\Big\vert + \Big\vert 
		\int_{\Onh}\u_t (1-J)\cdot\v_h\dif\x\Big\vert\\
	&\leqc \nrm{\nabla\u_t}{\Linf{\Odh{\Onh}}}\nrm{\Id-\Phi}{\Onh}
		\nrm{\v_h}{\Onh} + \nrm{J-1}{\Linf{\Odh{\Onh}}}\nrm{\u_t(t_n)}{\Onh}
		\nrm{\v_h}{\Onh}\\
	&\leqc h^{q+1}\nrm{\nabla\u_t}{\Linf{\Odh{\Onh}}}\nrm{\v_h}{\Onh} +
		h^q\nrm{\u_t(t_n)}{\Onh}\nrm{\v_h}{\Onh}\\
	&\leqc h^q\nrm{\u}{\pazocalbf{W}^{2,\infty}(\Q)}\nrm{\v_h}{\Onh}
\end{align*}
where we used \eqref{eqn:MappingEstimates} and
\begin{equation*}
	\Vert\u_t(\x,t_n)-(\u_t\circ\Phi)(\x,t_n)\Vert \leq
	\nrm{\nabla\u_t}{\Linf{\Odh{\Onh}}}\Vert \x - \Phi(\x) \Vert.
\end{equation*}
For the third and final term in \eqref{eqn:RewriteConsistencyT2} we have
\begin{align*}
	\Big\vert \int_{\Onh}\int_{t_{n-1}}^{t_n}\frac{t-t_{n-1}}{\dt}\u_{tt}\dif
		t\cdot\v_h\dif\x \Big\vert &\leq \int_{\Onh}
		\max_{s\in[t_{n-1},t_n]}\u_{tt}(s)\int_{t_{n-1}}^{t_n}
		\frac{t-t_{n-1}}{\dt}\dt\v_h\dif\x\\
	&\leq \frac{1}{2}\dt \nrm{\u_{tt}}{\Linf{\Odh{\Q}}}
		\nrm{\v_h}{\pazocalbf{L}^1(\Onh)}\\
	&\leqc \dt\nrm{\u}{\pazocalbf{W}^{2,\infty}(\Q)}\nrm{\v_h}{\Onh}.
\end{align*}
This then gives the bound
\begin{equation}\label{eqn:consistencyT2}
	\vert \mathfrak{T}_2 \vert \leqc (\dt + h^q)
	\nrm{\u}{\pazocalbf{W}^{2,\infty}(\Q)}\nrm{\v_h}{\Onh}.
\end{equation}

$\mathfrak{T}_3$: For the volume integral part in the diffusion consistency
error term $\mathfrak{T}_3$ we also refer to \cite[Lemma~7.4]{GOR15}. We
observe that it follows from the chain rule that
\begin{equation}\label{eqn:chain-rule--grad-vhl}
	\nabla\vhl(\hat\x) = \nabla(\v_h\circ\Phi^{-1})(\hat\x) = \nabla\v_h(\x)
	D\Phi(\x)^{-1}\qquad\text{for }\x\in\On,\;\x\coloneqq \Phi^{-1}(\hat\x).
\end{equation}
With this we can rewrite the volume contribution as
\begin{multline*}
	\int_{\Onh}\nu\nabla\u^n:\nabla\v_h\dif\x-\int_{\On}\nu\nabla\u^n
		:\nabla\vhl\dif\hat\x \\
	\begin{aligned}
		&=\int_{\Onh}\nu\nabla\u^n:\nabla\v_h\dif\x-\int_{\On}\nu
			\nabla\u^n:(\nabla\v_h)\circ\Phi^{-1} D\Phi^{-1}\dif\hat\x\\
		&=\int_{\Onh}\nu\nabla\u^n:\nabla\v_h\dif\x-\int_{\Onh}\nu
			\nabla\u^n\circ\Phi:J\nabla\v_h D\Phi^{-1}\dif\x \eqqcolon 
			\mathfrak{I}_3^1.
	\end{aligned}
\end{multline*}
First we observe that due to \eqref{eqn:MappingEstimates} we have
\begin{equation}\label{eqn:MappinEstimate2}
	\nrm{I-JD\Phi^{-1}}{\Linf{\Odh{\Onh}}} \leq
	\nrm{(I-IJ)D\Phi^{-1}}{\Linf{\Odh{\Onh}}} +
	\nrm{I-D\Phi^{-1}}{\Linf{\Odh{\Onh}}} \leqc h^q.
\end{equation}
With this and Lemma~\ref{lemma:MappingEstimates} we can estimate
$\vert\mathfrak{I}_3^1\vert$ as
\begin{align*}
	\vert\mathfrak{I}_3^1\vert &\leq \Big\vert \int_{\Onh}
		\nu(\nabla\u^n-\nabla\u^n\circ\Phi):J\nabla\v_h D\Phi^{-1}\dif\x 
		\Big\vert + \Big\vert \int_{\Onh}\nu\nabla\u^n:\nabla\v_h\cdot
		(I-JD\Phi^{-1})\dif\x\Big\vert\\
	&\leqc h^{q+1}\nrm{\u^n}{\H{2}{}{\Onh}}\nrm{\nabla\v_h}{\Onh} + h^q
		\nrm{\u^n}{\H{1}{}{\Onh}}\nrm{\nabla\v_h}{\Onh}\\
	&\leqc h^q \nrm{\u^n}{\pazocalbf{W}^{2,\infty}(\Q)}\nrm{\nabla\v_h}{\Onh}.
\end{align*}
We split the boundary contributions from $\mathfrak{T}_3=
a_h^n(\u^n,\v_h)-a_1^n (\u^n,\vhl)$ into three parts
\begin{equation*}
	\underbrace{-\int_{\Ghn}\nu\partial_{\n}\u^n\cdot\v_h\dif s +
		\int_{\G^n}\nu\partial_{\n}\u^n\cdot\vhl\dif\hat s}_{\mathfrak{I}_3^2}~
	\underbrace{-\int_{\Ghn}\nu\partial_{\n}\v_h\cdot\u^n\dif
		s}_{\mathfrak{I}_3^3}~
	\underbrace{+ \frac{\sigma}{h}\int_{\Ghn}\nu\u^n\cdot\v_h
		\dif s}_{\mathfrak{I}_3^4}.
\end{equation*}
For $\mathfrak{I}_3^2$ we use integration by substitution, the third
bound in \eqref{eqn:MappingEstimates} and Lemma~\ref{lemma:MappingEstimates},
which gives
\begin{multline*}
	\mathfrak{I}_3^2 = -\int_{\Ghn}\nu\partial_{\n}\u^n\cdot\v_h\dif s +
		\int_{\G^n}\nu\partial_{\n}\u^n\cdot\vhl\dif\hat s\\
	\begin{aligned}
		&= -\int_{\Ghn}\nu\partial_{\n}\u^n\cdot\v_h\dif s +
			\int_{\G^n_h}\nu\partial_{\n}\u^n\circ\Phi\cdot J\v_h\dif s\\
		&= \int_{\Ghn}\nu\partial_{\n}\u^n\cdot(J-1)\v_h\dif s +
			\int_{\Ghn}\nu(\partial_{\n}\u^n\circ\Phi-\partial_{\n}\u^n)\cdot 
			J\v_h\dif s\\
		&\leqc h^q\nrm{h^{\nf{1}{2}}\partial_{\n}\u^n}{\Ghn}
			\nrm{h^{-\nf{1}{2}}\v_h}{\Ghn}+h^{\nf{1}{2}}\nrm{\partial_{\n}\u^n
			\circ\Phi-\partial_{\n}\u^n}{\Ghn}\nrm{h^{-\nf{1}{2}}\v_h}{\Ghn}.
	\end{aligned}	
\end{multline*}
Using the standard trace inequality we have together with
Assumption~\ref{assumption:ExtensionProperties}
\begin{equation*}
	\nrm{\partial_{\n}\u^n}{\Ghn}\leqc \nrm{\u^n}{\H{2}{}{\Onh}}\leqc
	\nrm{\u^n}{\H{2}{}{\Od{\On}}} \leqc \nrm{\u^n}{\H{2}{}{\On}}.
\end{equation*}
Combining this result with \eqref{eqn:MappingEstLemma2:VelNormDeriv} now gives
the estimate
\begin{equation*}
	\vert \mathfrak{I}_3^2\vert \leqc
	h^q\nrm{\u^n}{\H{2}{}{\On}}\nrm{h^{-\nf{1}{2}}\v_h}{\Ghn} +
	\left[h^{m+\nf{1}{2}}\nrm{\u^n}{\H{m+1}{}{\On}} +
	h^q\nrm{\u^n}{\H{1}{}{\On}}\right]\nrm{h^{-\nf{1}{2}}\v_h}{\Ghn}
\end{equation*}

For the term $\mathfrak{I}_3^3$ we use the fact that $\u^n$ vanishes on $\G^n$,
so that with \eqref{eqn:chain-rule--grad-vhl}, \eqref{eqn:MappinEstimate2} and
\eqref{eqn:MappingEstimaLemma:L2Boundary} it follows
\begin{align*}
	\mathfrak{I}_3^3 = -\int_{\Ghn}\nu\partial_{\n}\v_h\cdot\u^n\dif s 
		&= \int_{\G^n}\nu\partial_{\n}\vhl\cdot\u^n\dif \hat s
		-\int_{\Ghn}\nu\partial_{\n}\v_h\cdot\u^n\dif s\\
	&= \int_{\Ghn}\nu J\partial_{\n}\v_h D\Phi^{-1}\cdot\u^n\circ\Phi\dif s
		-\int_{\Ghn}\nu\partial_{\n}\v_h\cdot\u^n\dif s\\
	&= \int_{\Ghn}\nu J\partial_{\n}\v_h D\Phi^{-1}\cdot(\u^n\circ\Phi -\u^n)
		\dif s +\int_{\Ghn}\nu\partial_{\n}\v_h(I-JD\Phi^{-1})\cdot\u^n\dif s\\
	&\leqc h^{q+\nf{1}{2}}\nrm{h^{\nf{1}{2}}\partial_{\n}\v_h}{\Ghn}
		\nrm{\u^n}{\H{2}{}{\On}} + h^q\nrm{h^{\nf{1}{2}}\partial_{\n}\v_h}{\Ghn}
		\nrm{h^{-\nf{1}{2}}\u^n}{\Ghn}.
\end{align*}
Using the trace estimate \eqref{eqn:TraceEstNormDerivBnd} we have for the first
boundary norm
\begin{equation*}
	\nrm{h^{\nf{1}{2}}\partial_{\n}\v_h}{\Ghn} =
	\sum_{T\in\ThGhn}\nrm{h^{\nf{1}{2}}\partial_{\n}\v_h}{T\cap\Ghn}\leqc
	\sum_{T\in\ThGhn}\nrm{\nabla\v_h}{T} \leqc \nrm{\nabla\v_h}{\OdT{n}}.
\end{equation*}
Furthermore, we can estimate $\nrm{h^{-\nf{1}{2}}\u^n}{\Ghn}$ using the fact
that $\u^n\circ\Phi(\x)=0$ for $\x\in\Ghn$ and
\eqref{eqn:MappingEstimaLemma:L2Boundary} to get
\begin{equation}\label{eqn:h-1/2u^n-BndBound}
	\nrm{h^{-\nf{1}{2}}u^n}{\Ghn}\leq
	h^{-\nf{1}{2}}\Big[\nrm{u^n-u^n\circ\Phi}{\Ghn} +
	\underbrace{\nrm{u^n\circ\Phi}{\Ghn}}_{=0}\Big]\leqc
	h^{q+\nf{1}{2}}\nrm{\u^n}{\H{2}{}{\On}}.
\end{equation}
These two estimates then give the bound
\begin{equation*}
	\vert\mathfrak{I}_3^3\vert \leqc h^q\nrm{\u^n}{\H{2}{}{\On}}
	\nrm{\nabla\v_h}{\OdT{n}}.
\end{equation*}
For the penalty term $\mathfrak{I}_3^4$ we proceed as for $\mathfrak{I}_3^3$
using \eqref{eqn:h-1/2u^n-BndBound}:
\begin{align*}
	\mathfrak{I}_3^4 = \frac{\sigma}{h}\int_{\Ghn}\nu\u^n\cdot\v_h\dif s 
		&= \frac{\sigma}{h}\int_{\Ghn}\nu\u^n\cdot\v_h\dif s -
		\frac{\sigma}{h}\int_{\G^n}\nu\u^n\cdot\vhl\dif \hat s\\
	&= \frac{\sigma}{h}\int_{\Ghn}\nu\u^n\cdot\v_h\dif s -
		\frac{\sigma}{h}\int_{\Ghn}\nu\u^n\circ\Phi\cdot J\v_h\dif s\\
	&= \frac{\sigma}{h}\int_{\Ghn}\nu(\u^n-\u^n\circ\Phi)\cdot J\v_h\dif s +
		\frac{\sigma}{h}\int_{\Ghn}\nu\u^n\cdot(1-J)\v_h\dif s\\
	&\leqc h^{q}\left[\nrm{\u^n}{\H{2}{ }{\On}} +
		\nrm{h^{-\nf{1}{2}}\u^n}{\Ghn}\right]\nrm{h^{-\nf{1}{2}}\v_h}{\Ghn}\\
		&\leqc h^{q}\nrm{\u^n}{\H{2}{ }{\On}}\nrm{h^{-\nf{1}{2}}\v_h}{\Ghn}
\end{align*}
Combining these estimates then gives the bound on $\mathfrak{T}_3$ as
\begin{align*}
	\vert\mathfrak{T}_3\vert &\leq \vert\mathfrak{I}_3^1\vert +
		\vert\mathfrak{I}_3^2\vert + \vert\mathfrak{I}_3^3\vert +
		\vert\mathfrak{I}_3^4\vert \\
	&\leqc h^q\nrm{\u}{\W{2,\infty}{\Q}}\tripnrmast{\v_h}{n} +
		h^{m+\nf{1}{2}}\nrm{\u^n}{\H{m+1}{}{\On}}\nrm{h^{-\nf{1}{2}}\v_h}{\Ghn}
		.\numberthis\label{eqn:consistencyT3}
\end{align*}

$\mathfrak{T}_4$: As in \eqref{eqn:chain-rule--grad-vhl}, for the divergence we
have
\begin{equation*}
	\nabla\cdot\vhl(\hat\x) = \trace(\nabla(\v_h\circ\Phi^{-1})(\hat\x)) =
	\trace(\nabla\v_h(\x) D\Phi(\x)^{-1})\qquad\text{for }\x\in\On,\;\x
	\coloneqq \Phi^{-1}(\hat\x).
\end{equation*}
We split the pressure term $\mathfrak{T}_4=\mathfrak{I}_4^1+\mathfrak{I}_4^2$
into volume and boundary integrals respectively. The volume term can then be
rewritten as
\begin{align*}
	\mathfrak{I}_4^1 &\coloneqq  -\int_{\Onh}p^n\nabla\cdot\v_h\dif\x +
		\int_{\On}p^n\nabla\cdot\vhl\dif\hat\x\\
	&= -\int_{\Onh}p^n\nabla\cdot\v_h\dif\x +
		\int_{\Onh}p^n\circ\Phi\,\trace(\nabla\v_h D\Phi^{-1})J\dif\hat\x\\
	&= 	\begin{multlined}[t]
			\int_{\Onh}p^n(J-1)\nabla\cdot\v_h\dif\x
				+ \int_{\Onh}(p^n\circ\Phi-p^n)\trace(\nabla\v_h D\Phi^{-1})
					J\dif\x\\
			+ \int_{\Onh}p^nJ\trace(\nabla\v_h D\Phi^{-1}-\nabla\v_h)\dif\x.
		\end{multlined}
\end{align*}
This can then be bounded by
\begin{align*}
	\vert\mathfrak{I}_4^1\vert &\leqc
	\begin{multlined}[t]
		h^q\nrm{p^n}{\Onh}\nrm{\nabla\cdot\v_h}{\Onh}
			+ \nrm{p^n\circ\Phi-p^n}{\Onh}
				\nrm{\trace(\nabla\v_h\,D\Phi^{-1})}{\Onh}\\
		+ \nrm{p^n}{\Onh}\nrm{\trace(\nabla\v_h D\Phi^{-1}-\nabla\v_h)}{\Onh}
	\end{multlined}\\
	&\leqc h^q\nrm{p^n}{\Onh}\nrm{\nabla\v_h}{\Onh} +
		h^{q+1}\nrm{p^n}{\pazocal{H}^1(\On)}\nrm{\nabla\v_h}{\Onh} +
		h^q\nrm{p^n}{\Onh}\nrm{\nabla\v_h}{\Onh}\\
 		&\leqc h^q \nrm{p^n}{\pazocal{H}^1(\On)}\nrm{\nabla\v_h}{\Onh}
\end{align*}
For the boundary terms we have
\begin{align*}
	\mathfrak{I}_4^2 &\coloneqq \int_{\Ghn}p^n\v_h\cdot\n\dif s -
		\int_{\G^n}p^n\vhl\cdot\n\dif\hat s\\
	&= \int_{\Ghn}p^n\v_h\cdot\n\dif s - \int_{\Ghn}p^n\circ\Phi\,\v_h\cdot\n 
		J\dif s\\
	&= \int_{\Ghn}p^n\v_h\cdot\n(1-J)\dif s + \int_{\Ghn}(p^n -
		p^n\circ\Phi)\,\v_h\cdot\n J\dif s\\
	&\leqc h^q \nrm{p^n}{\Ghn}\nrm{\v}{\Ghn} +
		\nrm{p^n-p^n\circ\Phi}{\Ghn}\nrm{\v}{\Ghn}
		\numberthis\label{eqn:consistecyI42}
\end{align*}
For the pressure part in the first of these terms two term we observe that
using \eqref{eqn:TraceIneqBnd} and
Assumption~\ref{assumption:ExtensionProperties}
gives us
\begin{align*}
	\nrm{p^n}{\Ghn}^2 = \sum_{K\in\ThGhn}\nrm{p^n}{K\cap\Ghn}^2 &\leqc
		\sum_{K\in\ThGhn}\left[h^{-1}\nrm{p^n}{K}^2 + 
		h\nrm{\nabla p^n}{K}^2\right] \\
	&\leqc h^{-1}\nrm{p^n}{\pazocal{H}^1(\Odh{\On})}^2 \\
	&\leqc h^{-1}\nrm{p^n}{\pazocal{H}^1(\On)}.
\end{align*}
For the pressure part in the second summand of \eqref{eqn:consistecyI42} we
apply \eqref{eqn:MappingEstLemma2:Pres}. Combining these results then gives
\begin{equation*}
	\vert\mathfrak{I}_4^2\vert \leqc
	h^q\nrm{p^n}{\pazocal{H}^1(\On)}\nrm{h^{-\nf{1}{2}}\v_h }{\Ghn} +
	h^m\nrm{p^n}{\pazocal{H}^m(\On)}\nrm{h^{-\nf{1}{2}}\v_h }{\Ghn}.
\end{equation*}
Together, these estimates then give us the bound
\begin{equation}\label{eqn:consistencyT4}
	\vert\mathfrak{T}_4\vert \leqc
	h^q\nrm{p^n}{\pazocal{H}^1(\On)}\tripnrmast{\v_h}{n} +
	h^m\nrm{p^n}{\pazocal{H}^m(\On)}\nrm{h^{-\nf{1}{2}}\v_h}{\Ghn}.
\end{equation}

$\mathfrak{T}_5$: The volume terms of the divergence constraint can again be
rewritten as
\begin{align*}
	\mathfrak{I}_5^1 &\coloneqq -\int_{\Onh}q_h\nabla\cdot\u^n\dif\x +
		\int_{\On}\qhl\nabla\cdot\u^n\dif\hat\x\\
	&=-\int_{\Onh}q_h\nabla\cdot\u^n\dif\x + \int_{\Onh}q_h\nabla\cdot
		\u^n\circ\Phi J\dif\hat\x\\
	&= \int_{\Onh}q_h\nabla\cdot\u^n(J-1)\dif\x +
		\int_{\Onh}q_h(\nabla\cdot\u^n\circ\Phi-\nabla\cdot\u^n)J\dif\x
\end{align*}
$\mathfrak{I}_5^1$ can then be bounded again using \eqref{eqn:MappingEstimates}
and Lemma~\ref{lemma:MappingEstimates}
\begin{align*}
	\vert\mathfrak{I}_5^1\vert &\leqc h^q\nrm{q_h}{\Onh}\nrm{\nabla\u^n}{\Onh} 
		+ \nrm{q_h}{\Onh}\nrm{\nabla\u^n\circ\Phi-\u^n}{\Onh}\\
	&\leqc h^q \nrm{q_h}{\Onh}\nrm{\nabla\u^n}{\Onh} +
		h^{q+1}\nrm{q_h}{\Onh}\nrm{\u^n}{\H{2}{}{\On}}\\
	&\leqc h^{q}\nrm{q_h}{\Onh}\nrm{\u^n}{\H{2}{}{\On}}.
\end{align*}
The boundary contribution in $\mathfrak{T_5}$ are
\begin{align*}
	\mathfrak{I}_5^2 &\coloneqq \int_{\Ghn}q_h\u^n\cdot\n\dif s -
		\int_{\G^n}\qhl\u^n\cdot\n\dif\hat s\\
	&=\int_{\Ghn}q_h\u^n\cdot\n\dif s - \int_{\Ghn}q_h(\u^n\cdot\n)\circ\Phi 
		J\dif s\\
	&=\int_{\Ghn}q_h(\u^n\cdot\n - (\u^n\cdot\n)\circ\Phi )J\dif s + 
		\int_{\Ghn}q_h\u^n\cdot\n(1-J)\dif s\\
	&\leqc \nrm{q_h}{\Ghn}\nrm{\u^n\cdot\n - (\u^n\cdot\n)\circ\Phi}{\Ghn} +
		h^q\nrm{q_h}{\Ghn}\nrm{\u^n\cdot\n}{\Ghn}	.
\end{align*}
Using the trace estimate \eqref{eqn:TraceIneqBnd} and the inverse estimate
\eqref{eqn:InvEst:DerivEl} gives us
\begin{equation*}
	\nrm{q_h}{\G_h^n} \leqc h^{-\nf{1}{2}}\nrm{q_h}{\OdT{n}}.
\end{equation*}
For the term $\nrm{\u^n\cdot\n - (\u^n\cdot\n)\circ\Phi}{\Ghn}$ we can apply
Lemma~\ref{lemma:MappingEstimate2}. Furthermore, we have due to the homogeneous
Dirichlet conditions that
\begin{align*}
	\nrm{\u^n\cdot\n}{\Ghn} &\leq \nrm{\u^n\cdot\n-(\u^n\cdot n)\circ\Phi}{\Ghn} 
		+ \nrm{\u^n\cdot\n\circ\Phi}{\Ghn}\\
	&\leq \nrm{\u^n\cdot\n-(\u^n\cdot n)\circ\Phi}{\Ghn} + 
		\underbrace{\nrm{\u^n}{\G^n}}_{=0}\\
	&\leqc h^{q+\nf{1}{2}}\nrm{\u^n}{\H{1}{}{\On}} + 
		h^{m+\nf{1}{2}}\nrm{\u^n}{\H{m+1}{}{\On}}
\end{align*}
combining these results then gives us the bound
\begin{equation*}
	\vert\mathfrak{I}_5^2\vert \leqc (1+h^q)\left[h^q\nrm{\u^n}{\H{1}{}{\On}} 
	+ h^m\nrm{\u^n}{\H{m+1}{}{\On}} \right]\nrm{q_h}{\OdT{n}}.
\end{equation*}
These bound on the volume and boundary contributions of $\mathfrak{T}_5$ then
give
\begin{equation}\label{eqn:consistencyT5}
	\vert\mathfrak{T}_5\vert \leqc h^q \nrm{\u^n}{\H{2}{}{\On}}\tripnrmast{q_h}{n}
	+ h^m\nrm{\u^n}{\H{m+1}{}{\On}}\tripnrmast{q_h}{n}.
\end{equation}

$\mathfrak{T}_6$: To bound $\vert\mathfrak{T}_6\vert$ we use the
Cauchy-Schwartz inequality, Lemma~\ref{lemma:GP-consitency} and
Assumption~\ref{assumption:ExtensionProperties} to get
\begin{align*}
	\vert\mathfrak{T}_6\vert 
	&= \gamma_s\vert L \nu i_h^n(\u^n,\v_h)
		+ \nf{L}{\nu} i_h^n(\u^n,\v_h) 
		+ \nf{1}{\nu}j_h^n(p^n,q_h)\vert\\
	&\leqc (\nu + \nf{1}{\nu})L^{\nf{1}{2}}i_h^n(\u^n,\u^n)^{\nf{1}{2}}
		L^{\nf{1}{2}}i_h^n(\v_h,\v_h)^{\nf{1}{2}}
		+ j_h^n(p^n,p^n)^{\nf{1}{2}}j_h^n(q_h,q_h)^{\nf{1}{2}} \\
	&\leqc (\nu + \nf{1}{\nu}) h^m L^{\nf{1}{2}}
		\nrm{\u^n}{\H{m+1}{}{\OdT{n}}}\nrm{\v_h}{\H{1}{}{\OdT{n}}}
		+ h^m\nrm{p^n}{\pazocal{H}^m(\OdT{n})}\nrm{q_h}{\OdT{n}}\\
	&\leqc (\nu + \nf{1}{\nu}) h^m L^{\nf{1}{2}}
		\nrm{\u^n}{\H{m+1}{}{\On}}\tripnrmast{\v_h}{n}
		+ h^m \nrm{p^n}{\pazocal{H}^m(\On)}\tripnrmast{q_h}{n}.
		\numberthis\label{eqn:consistencyT6}
\end{align*}

Combining estimates \eqref{eqn:consistencyT1}, \eqref{eqn:consistencyT2},
\eqref{eqn:consistencyT3}, \eqref{eqn:consistencyT4}, \eqref{eqn:consistencyT5}
and \eqref{eqn:consistencyT6} then gives the desired result.
\end{proof}


\subsection{Error Estimates}\label{sec:analysis::subsec:error-estimates}

We define $\u^n_I\coloneqq \Int\u^n$ and $p^n_I\coloneqq \Int p^n$. We then
split the velocity error $\Eu{n}$ and the pressure $\Ep{n}$ each into an
interpolation and discretisation error
\begin{alignat*}{3}
	& \Eu{n} &&= (\u^n - \u^n_I) + (\u^n_I - \u^n_h) &&= \etab^n + \erru{n}\\
	& \Ep{n} &&= (p^n -p^n_I) + (p^n_I - p^n_h) &&= \zeta^n + \errp{n}. 
\end{alignat*}
Inserting this into the error equation \eqref{eqn:ErrorEquation} and
rearranging terms then yields
\begin{multline}\label{eqn:DiscretisationErrorEquation}
	\int_{\Onh}\frac{\erru{n}-\erru{n-1}}{\dt}\cdot\v_h\dif\x + 
		a_n^h(\erru{n},\v_h) + b_h^n(\errp{n},\v_h) + b_h^n(q_h,\erru{n}) +		
		s_h^n((\erru{n},\errp{n}),(\v_h,q_h)) \\
	= \mathfrak{E}_c^n(\v_h,q_h) + \mathfrak{E}_I^n(\v_h,q_h)
\end{multline}
with 
\begin{equation*}
	\mathfrak{E}_I^n(\v_h,q_h) = -\int_{\Onh}\frac{\etab^n-\etab^{n-1}}{\dt}
	\cdot\v_h\dif\x - a_n^h(\etab^n,\v_h) - b_h^n(\zeta^n,\v_h) - 
	b_h^n(q_h,\etab^n) - s_h^n((\etab^n,\zeta^n),(\v_h,q_h)).
\end{equation*}

For the interpolation error component, we can prove the following estimate.

\begin{lemma}\label{lemma:InterpolationErrorComponent}
Assume for the velocity that $\u\in\L{\infty}{0,T;\H{m+1}{}{\O(t)}}$,
$\u_t\in\L{\infty}{0,T;\H{m}{}{\O(t)}}$ and for the pressure that
$p\in\pazocal{L}^{\infty}(0,T;\pazocal{H}^m(\O(t)))$. We can then bound the
interpolation error term by
\begin{equation*}
	\vert\mathfrak{E}_I^n(\v_h,q_h)\vert \leqc \frac{h^m L^{\nf{1}{2}}}{\nu} 
	R_{I,1}(\u,p)\tripnrmast{\v_h}{n} + h^m  R_{I,2}(\u,p) \tripnrmast{q_h}{n}.
\end{equation*}
with
\begin{equation*}
	R_{I,1}(\u,p) = \sup_{t\in[0,T]}\left( \nrm{\u}{\H{m+1}{}{\O(t)}} +
		\nrm{\u_t}{\H{m}{}{\O(t)}} + \nrm{p}{\pazocal{H}^m(\O(t))}\right)
\end{equation*}
and
\begin{equation*}
	R_{I,2}(\u,p) = \sup_{t\in[0,T]}\left( \nrm{\u}{\H{m+1}{}{\O(t)}} +
		\nrm{p}{\pazocal{H}^m(\O(t))}\right)
\end{equation*}

\end{lemma}

\begin{proof}
We split the interpolation error terms into five different parts
$\mathfrak{E}_I^n(\v_h,q_h) =
\mathfrak{T}_7+\mathfrak{T}_8+\mathfrak{T}_9+\mathfrak{T}_{10}+\mathfrak{T}_{11
}$. These are the time-derivative term, the diffusion bilinear form, the
pressure coupling term, the divergence constraint and ghost-penalty operator
respectively. As in Lemma~\ref{lemma:consistency}, we deal with each
constituent term separately.
	
The proof for the time-derivative approximation $\mathfrak{T}_7$ is part of
\cite[Lemma~5.12]{LO19}. For completeness sake, we also give it here. We extent
$\u_h^n$ in time as the Lagrange interpolant of $\u(t)$ in all nodes of
$\OdT{n}$, so that $\u^n_I(t_{n-1})=\u^{n-1}_I$ on $\Onh$. As $(\u^n_I)_t$ is
the Lagrange interpolant of $\u_t$, we then have with
\eqref{eqn:DiscreteDomainInclusions} and \eqref{eqn:ExtentionExtimate} that
\begin{equation}\label{eqn:InterpolTimederivativeEst}
	\nrm{\etab^n_t}{\Onh}\leqc h^m\nrm{\u_t}{\H{m}{}{\OdT{n}}}\leqc
	h^m\nrm{\u_t}{\Od{\O(t)}}\leqc h^m\nrm{\u_t}{\H{m}{}{\O(t)}}\quad t\in I_n.
\end{equation}
	
Using the Cauchy-Schwarz inequality twice as well as
\eqref{eqn:InterpolTimederivativeEst} and
\eqref{eqn:TimeDerivExtentionExtimate} we then get
\begin{align*}
	\vert\mathfrak{T}_7\vert = \left\vert
		\int_{\Onh}\frac{\etab^n-\etab^{n-1}}{\dt}\cdot\v_h\dif\x\right\vert 
		&\leq\vert\dt\vert^{-1}\nrm{\etab^n-\etab^{n-1}}{\Onh}
		\nrm{\v_h}{\Onh}\\
	&= \vert\dt\vert^{-1}\Big\Vert\int_{t_{n-1}}^{t_n} \etab_t(s)\dif
		s\Big\Vert_{\Onh}\nrm{\v_h}{\Onh}\\
	&\leq \vert\dt\vert^{-\nf{1}{2}}\Big(
		\int_{t_{n-1}}^{t_n}\nrm{\etab_t(s)}{\Onh}\dif
		s\Big)^{\nf{1}{2}}\nrm{\v_h}{\Onh}\\
	&\leq \sup_{t\in[t_{n-1},t_n]}\nrm{\etab_t(t)}{\Onh}\nrm{\v_h}{\Onh}\\
	&\leqc h^m \sup_{t\in[t_{n-1},t_n]}\nrm{u_t(t)}{\H{m}{}{\Od{\O(t)}}}
		\nrm{\v_h}{\Onh}\\
	&\leqc h^m\sup_{t\in[0,T]}\big( \nrm{\u}{\H{m+1}{}{\O(t)}} +
		\nrm{\u_t}{\H{m}{}{\O(t)}} \big)\nrm{\v_h}{\Onh}.
\end{align*}
	
For the diffusion term $\mathfrak{T}_8$ we use the continuity result
\eqref{eqn:DiffusionContinuity} and Lemma~\ref{lemma:InterpolationEstTripNrm}
for the interpolation term and \eqref{eqn:TripNormEst} for the test function,
which gives
\begin{equation*}
	\vert\mathfrak{T}_8\vert = \vert -a_h^n(\etab^n,\v_h)\vert \leqc
	\tripnrm{\etab^n}_n\tripnrm{\v_h}_n\leqc
	h^m\nrm{\u^n}{\H{m+1}{}{\On}}\tripnrmast{\v_h}{n}.
\end{equation*}
using the same technique, we can estimate the pressure and divergence bilinear
forms as
\begin{alignat*}{3}
	& \vert \mathfrak{T}_9\vert &&= \vert -b_h^n(\zeta^n,\v_h)\vert &&\leqc
		h^m\nrm{p^n}{\pazocal{H}^m(\On)}\tripnrmast{\v_h}{n}\\
	& \vert\mathfrak{T}_{10}\vert &&= \vert -b_h^n(q_h,\etab^n)\vert &&\leqc
		h^m\nrm{\u^n}{\H{m+1}{}{\On}}\tripnrmast{q_h}{n}.
\end{alignat*}
	
For the ghost-penalty term $\mathfrak{T}_{11}$ (see also
\cite[Lemma~5.12]{LO19}) we use the Cauchy-Schwarz inequality and
Lemma~\ref{lemma:GP-consitency} with \eqref{eqn:ExtentionExtimate}
\begin{align*}
	\vert\mathfrak{T}_{11}\vert = \vert s_h^n((\etab^n,\zeta^n),(\v_h,q_h)) 
		&\leqc (\nu+\nf{1}{\nu}) L^{\nf{1}{2}}
			i_h^n(\etab^n,\etab^n)^{\nf{1}{2}}
			L^{\nf{1}{2}} i_h^n(\v_h,\v_h)^{\nf{1}{2}}
		+ j_h^n(\zeta^n,\zeta^n)^{\nf{1}{2}}j_h^n(q_h,q_h)^{\nf{1}{2}}\\
	&\leqc \begin{multlined}[t]
 				(\nu + \nf{1}{\nu}) h^m L^{\nf{1}{2}}
				\nrm{\u^n}{\H{m+1}{}{\OdT{n}}}\nrm{\v_h}{\OdT{n}}\\
				+ h^m\nrm{p^n}{\pazocal{H}^m(\OdT{n})}\nrm{q_h}{\OdT{n}}
		 	\end{multlined}\\
	&\leqc (\nu + \nf{1}{\nu}) h^m L^{\nf{1}{2}}\nrm{\u^n}{\H{m+1}{}{\On}}
			\tripnrmast{\v_h}{n} +
		h^m\nrm{p^n}{\pazocal{H}^m(\On)}\tripnrmast{q_h}{n}.
\end{align*}

\end{proof}

\begin{theorem}\label{theorem:ErrorEstimateEnergyNorm}
For sufficiently small $\dt$ and $h$, the velocity error can be bound by
\begin{multline*}
	\nrm{\Eu{n}}{\Onh}^2 
		+ \sum_{k=1}^n\Big\{\nrm{\Eu{k}-\Eu{k-1}}{\O_h^k}^2 
		+ \dt\big[ \nu c_{L\ref{lemma:coercivity}}
			\tripnrmast{\Eu{k}}{k}^2 + \frac{L}{\nu}i_h^k(\Eu{k},\Eu{k}) 
			\big] \Big\}\\
	\leq \exp((c_{T\ref{theorem:ErrorEstimateEnergyNorm}a}/\nu) t_n)\left[
		+ \dt\sum_{k=1}^n c_{T\ref{theorem:ErrorEstimateEnergyNorm}b} \Big[ \dt^2 +
			h^{2q} + \frac{h^{2m}L}{\nu} + \frac{1}{\dt}(h^{2q} +
			\frac{h^{2m}}{\nu})\Big] R(\u,p,\f) \right]
\end{multline*}
with $R(\u,p,\f)= \sup_{t\in[0,T]}(\nrm{\u}{\H{m+1}{}{\O(t)}}^2 +
\nrm{\u_t}{\H{m}{}{\O(t)}}^2 + \nrm{p}{\pazocal{H}^m(\O(t))}^2) +
\nrm{\u}{\W{2,\infty}{\Q}}^2 + \nrm{\f}{\H{1}{}{\On}}^2$
and constants $c_{T\ref{theorem:ErrorEstimateEnergyNorm}}$ independent of
$\dt$, $n$ and $h$. 

For the pressure we have the bound
\begin{equation*}
	\dt^2\sum_{k=1}^n \tripnrmast{\Ep{k}}{k}^2
	\leqc \dt\sum_{k=1}^{n-1}\frac{1}{\nu} \nrm{\Eu{k}}{\O_h^{k}}^2
		+ \dt\sum_{k=1}^n c \Big[ \dt^2 + h^{2q} + \frac{h^{2m}L}{\nu} +
			\frac{1}{\dt}(h^{2q} + \frac{h^{2m}}{\nu})\Big] R(\u,p,\f).
\end{equation*}
\end{theorem}

\begin{proof}
We prove the result for the discretisation error, since the result then
immediately follows by optimal interpolation properties.
We start with the velocity estimate. Similar to the stability proof, for
$n=k$ we test the error equation \eqref{eqn:DiscretisationErrorEquation} with
the test-function $(\v_h,q_h)=2\dt(\erru{k},-\errp{k})$ and use the identity
\eqref{eqn:TimeDiffIdentity} to get
\begin{multline*}
	\nrm{\erru{k}}{\O_n^k}^2 + \nrm{\erru{k} -\erru{k-1}}{\O_h^k}^2  
		+ 2\dt(a_h^k + \nu L i_h^k)(\erru{k},\erru{k}) 
		+ \nf{2 L \dt}{\nu} i_h^k(\erru{k},\erru{k})
		+ \nf{2\dt}{\nu}j_h^k(\errp{k},\errp{k}) \\
	= 2\dt(\mathfrak{E}_c^k+ \mathfrak{E}_I^k)(\erru{k},-\errp{k}) 
		+ \nrm{\erru{k-1}}{\O_h^k}^2.
\end{multline*}
Using the coercivity result Lemma~\ref{lemma:coercivity} and
\eqref{eqn:estimate-on-old-domain} we get (with the appropriate choice of
$\eps$) that
\begin{multline*}
	\nrm{\erru{k}}{\O_n^k}^2 + \nrm{\erru{k} -\erru{k-1}}{\O_h^k}^2  
		+ 2\dt c_{L\ref{lemma:coercivity}}\tripnrmast{\erru{k}}{k}^2
		+ \dt\frac{2 L}{\nu} i_h^k(\erru{k},\erru{k})
		+ \dt\frac{2}{\nu}j_h^k(\errp{k},\errp{k}) \\
	\leq (1 + \frac{\overline c}{\nu} \dt)\nrm{\erru{k-1}}{\O_h^{k-1}}^2
		+ \dt\frac{\nu c_{L\ref{lemma:coercivity}}}{2}\tripnrmast{\erru{k-1}}{k-1}^2
		+ \dt\frac{c'h^2}{\nu} L i_h^{k-1}(\erru{k-1}, \erru{k-1})\\
		+ 2\dt(\vert\mathfrak{E}_c^k \vert + \vert\mathfrak{E}_I^k\vert)
			(\erru{k},\errp{k}).
\end{multline*}
Applying the weighted Young inequality to Lemma~\ref{lemma:consistency}
and Lemma~\ref{lemma:InterpolationErrorComponent} then gives
\begin{multline*}
	\vert\mathfrak{E}_c^k(\erru{k},\errp{k}) +
		\mathfrak{E}_I^k(\erru{k},\errp{k})\vert \\
	\leq \frac{1}{\eps_1}c\big(\dt^2 + h^{2q} 
		+ \frac{h^{2m}L}{\nu^2}\big)R(\u,p,\f)
		+ \eps_1\tripnrmast{\erru{k}}{k}^2 
		+ \frac{1}{\eps_2}c\big(h^{2q} + \frac{h^{2m}}{\nu^2}\big)R'(\u,p) 
		+ \eps_2\tripnrmast{\errp{k}}{k}^2
\end{multline*}
with
\begin{equation*}
	R'(\u,p) = \sup_{t\in[0,T]}\big(
		\nrm{\u(t)}{\H{m+1}{}{\O(t)}}^2 + \nrm{p(t)}{\pazocal{H}^m(\O(t))}^2 \big).
\end{equation*}
Now we choose $\eps_1 = \nf{\nu c_{L\ref{lemma:coercivity}}}{4}$ and $\eps_2 =
\nf{\dt\beta^2}{4c_Yc_{P,h}^2}$. With the constant $c_Y>0$ to be specified
later. Inserting these bound on the consistency and interpolation estimates
into the above inequality and summing over $k=1,\dots,n$ and using
$\erru{0}=\bm{0}$ gives
\begin{multline}\label{eqn:Vel-error-estimate-intermediate}
	\nrm{\erru{n}}{\O_h^k}^2 
		+ \sum_{k=1}^n\nrm{\erru{k}-\erru{k-1}}{\O_h^k}^2 
		+ \dt \sum_{k=1}^n\big[ \nu c_{L\ref{lemma:coercivity}}
			\tripnrmast{\erru{k}}{k}^2 + \frac{L}{\nu}i_h^k(\erru{k},\erru{k}) 
			\big]
		+ \dt \sum_{k=1}^n \frac{2}{\nu} j_h^k(\errp{k},\errp{k}) \\
	\leq \dt\sum_{k=1}^{n-1}\frac{\overline{c}}{\nu} \nrm{\erru{k}}{\O_h^{k}}^2
		+ \dt^2 \sum_{k=1}^2 \frac{\beta^2}{2c_Yc_{P,h}^2}
			\tripnrmast{\errp{k}}{k}^2\\
		+ \dt\sum_{k=1}^n c \Big[ \dt^2 + h^{2q} + \frac{h^{2m}L}{\nu} +
			\frac{1}{\dt}(h^{2q} + \frac{h^{2m}}{\nu})\Big] R(\u,p,\f).
\end{multline}
under the assumption, that $h$ is sufficiently small, such that $c'h^2 \leq 1$.
To complete the velocity estimate, we therefore require the pressure
estimate.

Rearranging the error equation \eqref{eqn:DiscretisationErrorEquation} and
using the test-function $q_h=0$ gives
\begin{align*}
	b_h^k(\errp{k},\v_h) 
		&= - \lprod{\nf{1}{\dt}(\erru{k}-\erru{k-1})}{\v_h}{\O_h^k}
			-(a_h^k + \nu L i_h^k)(\erru{k},\v_h)
			- \frac{L}{\nu}i_h^k(\erru{k},\v_h)
			+ (\mathfrak{E}_c^k+ \mathfrak{E}_I^k)(\erru{k},0)\\
		&\leq 
		\begin{multlined}[t]
		 	\Big[ \frac{c_{P,h}}{\dt}\nrm{\erru{k}-\erru{k-1}}{\O_h^k}
			+ \nu c_{L\ref{lemma:continuity}}\tripnrmast{\erru{k}}{k}
			+ \frac{L^{\nf{1}{2}}}{\nu}i_h^k(\erru{k},\erru{k})^{\nf{1}{2}}\\
			+ \hat{c}\big(\dt + h^q + \frac{L^{\nf{1}{2}}h^m}{\nu}\big)
				(R_{c,1}+R_{I,1})(\u,p,\f) \Big]\tripnrmast{\v_h}{k}
		\end{multlined}
\end{align*}
where $ \hat{c} = c_{L\ref{lemma:consistency}} + 
c_{L\ref{lemma:InterpolationErrorComponent}}$. Using the inf-sup result from
Lemma~\ref{lemma:bad-inf-sup} together with \eqref{eqn:PreNormEqiv}, we then
have
\begin{multline*}
	\beta\tripnrmast{\errp{k}}{k} \leq
		\frac{c_{P,h}}{\dt}\nrm{\erru{k}-\erru{k-1}}{\O^k_h}  
		+ \nu c_{L\ref{lemma:continuity}}\tripnrmast{\erru{k}}{k}
		+ \frac{L^{\nf{1}{2}}}{\nu}i_h^k(\erru{k},\erru{k})^{\nf{1}{2}}
		+ (1+\beta)j_h^n(\errp{k},\errp{k})^{\nf{1}{2}}\\
		+ \hat{c}\big(\dt + h^q + \frac{L^{\nf{1}{2}}h^m}{\nu}\big)
			(R_{c,1}+R_{I,1})(\u,p,\f)
\end{multline*}
Squaring this, using Young's inequality to remove the product terms multiplying
with $\dt^2$ and summing over $k=1,\dots,n$, we get
\begin{multline}\label{eqn:pres_discr-err-intermediate_est}
	\dt^2\sum_{k=1}^n\frac{\beta^2}{c_{P,h}^2c_Y}\tripnrmast{\errp{k}}{k}^2
	\leq \sum_{k=1}^n\nrm{\erru{k}-\erru{k-1}}{\O^k_h}^2
		+ \dt\sum_{k=1}^n \Big[\dt(\nu^2c_{L\ref{lemma:continuity}}^2)
			\tripnrmast{\erru{k}}{k}^2
			+ \frac{\dt L}{\nu^2}i_h^k(\erru{k},\erru{k}) \Big]\\
		+ \dt \sum_{k=1}^n \dt(1+\beta)^2 j_h^n(\errp{k},\errp{k})
		+ \dt \sum_{k=1}^n \dt \hat{c}^2\big(\dt^2 + h^{2q} +
			\frac{Lh^{2m}}{\nu^2}\big)R(\u,p,\f)
\end{multline}
where $c_Y$ stems from the estimate $(\sum_{i=1}^n a_i)^2 \leq
n\sum_{i=1}^na_i^2$. We make the technical assumptions that
$\dt\nu^2c_{L\ref{lemma:continuity}}^2 \leq \nu c_{L\ref{lemma:continuity}}$
and $\dt(1+\beta^2)\leq 2/\nu$. Note, that since we are interested in the case 
of $\nu \ll 1$, these assumptions are not problematic and the inequalities are
not sharp. For sufficiently small $\dt$, that is $\dt L/\nu^2\leq 2L/\nu$, we
can then bound the right-hand side of
\eqref{eqn:pres_discr-err-intermediate_est} with
\eqref{eqn:Vel-error-estimate-intermediate} which proves the error estimate.

The pressure discretisation error on the right hand side of 
\eqref{eqn:Vel-error-estimate-intermediate} can therefore be bound by the
other terms on the right-hand side of
\eqref{eqn:Vel-error-estimate-intermediate}. This then gives us the estimate
\begin{multline}
	\nrm{\erru{n}}{\O_h^k}^2 
		+ \sum_{k=1}^n\Big\{\nrm{\erru{k}-\erru{k-1}}{\O_h^k}^2 
		+ \dt\big[ \nu c_{L\ref{lemma:coercivity}}
			\tripnrmast{\erru{k}}{k}^2 + \frac{L}{\nu}i_h^k(\erru{k},\erru{k}) 
			\big] \Big\}\\
	\leq \dt\sum_{k=1}^{n-1}\frac{2\overline{c}}{\nu}\nrm{\erru{k}}{\O_h^{k}}^2
		+\dt\sum_{k=1}^n 2c \Big[ \dt^2 + h^{2q} + \frac{h^{2m}L}{\nu} +
			\frac{1}{\dt}(h^{2q} + \frac{h^{2m}}{\nu})\Big] R(\u,p,\f).
\end{multline}
Applying Gronwall's Lemma then proves the result.
\end{proof}

\begin{remark}
To be able to get an optimal estimate for the velocity and pressure error from
the above proof, which do not depend on negative powers of $\dt$ we would need
an estimate for $\nf{1}{\dt}\nrm{\erru{k} - \erru{k-1}}{\O^k_h}^2$ which is
independent of negative powers of $\dt$.
Standard approaches to estimate this, such as in \cite{BW11}, do not work here,
since the proof requires $\erru{k} - \erru{k-1}$ to be weakly divergence free
with respect to $Q_h^k$. An alternative would be to estimate
$\nrm{\frac{\erru{k} -\erru{k-1}}{\dt}}{-1}$. However, standard approaches such
as in \cite{dFGJN15} are again not viable due to lack of weak divergence
conformity of the solution at two different time-steps.

If such an estimate was at hand, we would also be able to avoid the negative
power of $\dt$ on the right hand side of both the velocity and pressure
estimates.
\end{remark}

\begin{remark}
The choice of $\dt\sim h$ would lead to an overall loss of half an order in the
energy norm, provided that the spatial error is dominant. However, since this
term
is the result of lack of full consistency of the
method (due to geometry approximation, ghost-penalties and divergence
constraint with respect to different spaces for different time-steps), we do
not expect that this is the major, dominating error part.

\end{remark}

\begin{remark}
It cannot be expected, that the above result is sharp, since we have made some
very crude estimates to bound \eqref{eqn:pres_discr-err-intermediate_est} with
\eqref{eqn:Vel-error-estimate-intermediate}.
\end{remark}

\begin{remark}[Extension to higher-order BDF schemes]
As remarked upon in \cite{LO19,BFM19}, the extension to higher oder BDF methods
such as BDF2 is relatively unproblematic. 
\end{remark}


\section{Numerical Examples}
\label{sec:numerical-examples}

The method has been implemented using \texttt{ngsxfem} \cite{ngsxfem}, an
add-on package to the high-order Finite-Element library
\texttt{NGSolve}/\texttt{Netgen} \cite{Sch97,Sch14}. The resulting sparse
linear systems are solved using the direct solver \texttt{Pardiso}, as part of
the \texttt{Intel MKL} library \cite{IntelMKL}.

Full datasets of the computations shown here, as well as scripts to reproduce
the results can be found online: DOI:
\href{https://doi.org/10.5281/zenodo.3647571}{\texttt{10.5281/zenodo.3647571}}.

\subsection{General Set-up}
\label{sec:numerical-examples::subsec:set-up}

As a test case, we consider a basic test case with an analytical solution. 
On the domain $\O(t)=\{\x\in\R^2\st (\x_1-t)^2 + \x_2^2 = \nf{1}{2}\}$ we
consider the analytical solution 
\begin{equation*}
	\u_{ex}(t) = \begin{pmatrix}
		2\pi\x_2\cos(\pi((\x_1-t)^2 + \x_2^2))\\
		-2\pi\x_1\cos(\pi((\x_1-t)^2 + \x_2^2))
	\end{pmatrix}
	\quad\text{and}\quad 
	p_{ex}(t)=\sin(\pi((\x_1-t)^2 + \x_2^2)) - \nf{2}{\pi}.
\end{equation*}
The velocity $\u_{ex}(t)$ then fulfils homogeneous Dirichlet boundary
conditions and we have $p_{ex}(t)\in\Ltwozero{\O_n(t)}$. An illustration of the
initial solution on the initial background mesh can be seen in 
Figure~\ref{fig:AnalyticalSolution}.
The forcing vector is then chosen accordingly as $\f(t)= \partial_t\u_{ex}(t)
-\nu\Delta\u_{ex}(t) + \nabla p_{ex}(t)$. 
 
 \begin{figure}
 	\centering
 	\includegraphics[width=0.4\textwidth]{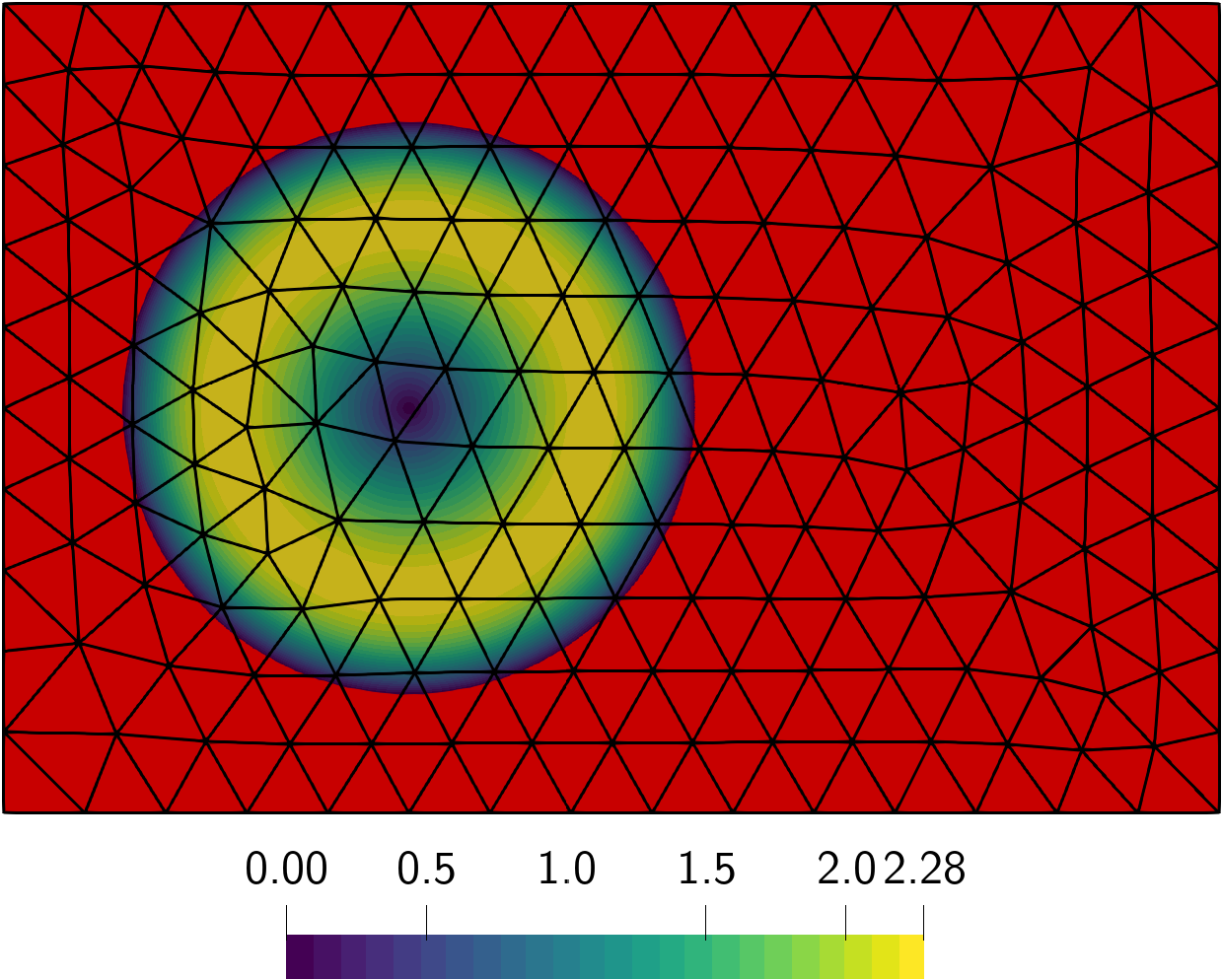}
 	\hskip 0.05\textwidth
 	\includegraphics[width=0.4\textwidth]{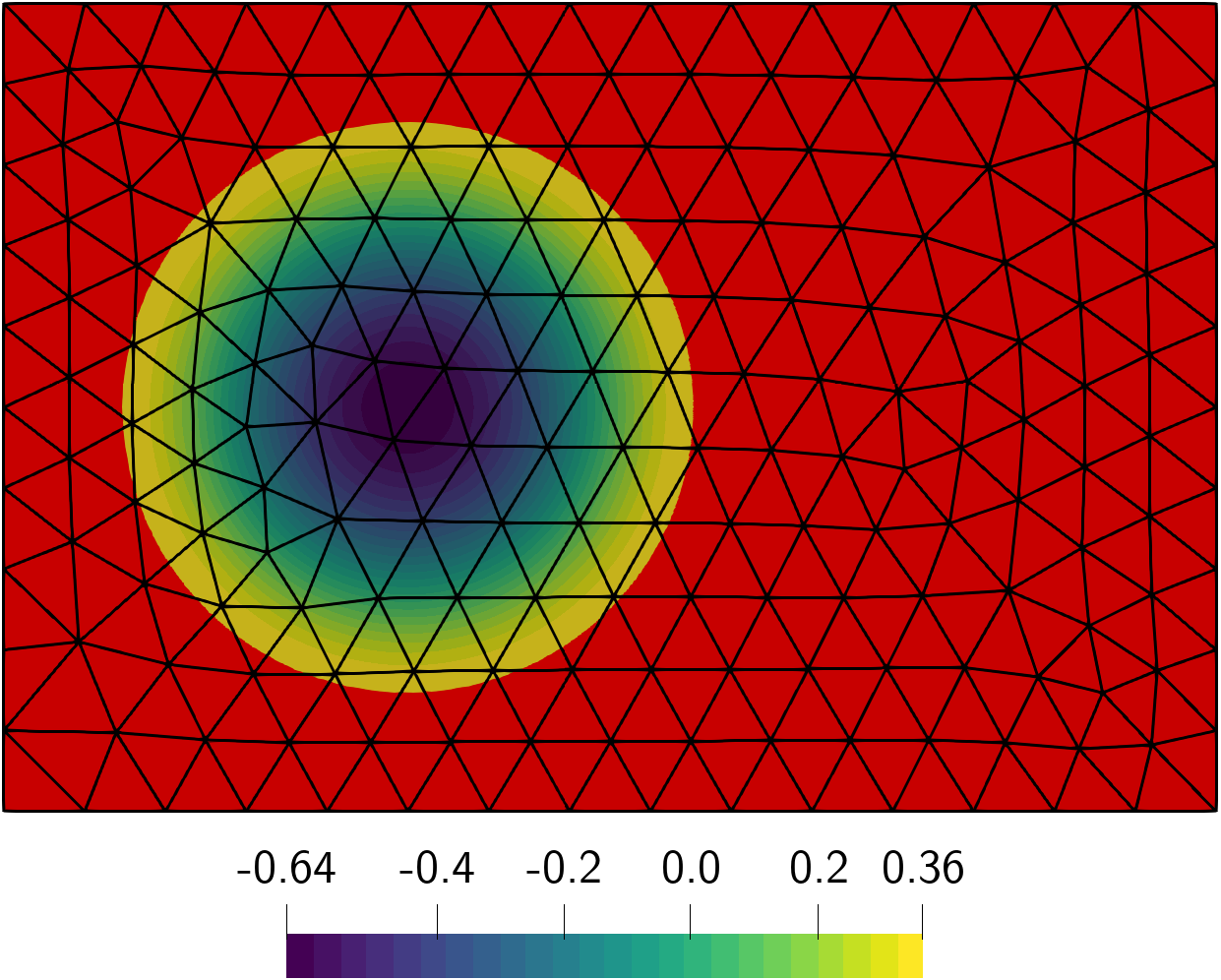}
	\caption{Velocity magnitude and pressure of the analytical solution 
		at $t=0$.}
 	\label{fig:AnalyticalSolution}
 \end{figure}
 
We then consider the following set-up: The time interval is $[0,1]$ and the
background domain is taken as $\Ot = (-1,-2)\times(1,1)$. The maximum interface
speed in our time interval is then $w^n_{\infty}=1$ and for the strip-width we
choose $c_\delta = 1$. Unless otherwise stated, we take the Nitsche penalty
parameter as $\sigma=40k^2$. On the mesh, we consider the lowest order
Taylor-Hood elements, i.e., $k=2$. However, as the geometry is approximated by
a piecewise linear level set function we have $q=1$, so that we cannot expect
to get the full spatial convergence order from the elements.

To quantify the computational results, we will consider the following discrete
space-time errors:
\begin{gather*}
	\nrm{\u_h - \u}{\ell^2(\pazocalbf{L}^2)}^2 \coloneqq \dt\sum_{k=1}^n
		\nrm{\u_h-\u}{\pazocalbf{L}^2(\O_h^k)}^2\qquad
	\nrm{\u_h - \u}{\ell^2(\pazocalbf{H}^1)}^2 \coloneqq \dt\sum_{k=1}^n
		\nrm{\nabla\u_h-\nabla\u}{\pazocalbf{L}^2(\O_h^k)}^2\\
	\nrm{p_h - p}{\ell^2(\pazocalbf{L}^2)}^2 \coloneqq \dt\sum_{k=1}^n
		\nrm{p_h-p}{\pazocal{L}^2(\O_h^k)}^2.
\end{gather*}

\subsection{Parameter dependence}
\label{sec:numerical-examples::subsec:parameter-dependence}

We investigate the robustness of our method with respect to the viscosity
$\nu$, the ghost-penalty stabilisation parameter $\gamma_s$ and the extension
strip-width, in terms of the number of elements $L$. For this we take the
time-step $\dt=0.05$. The viscosity is taken as $\nu\in\{10^0, 10^{-1}, \dots,
10^{-4}\}$ and the stabilisation is taken as $\gamma_s\in\{0.1, 1,
\dots, 10^3\}$. To increase the strip-width $L = \lceil \nf{\delta_h}{h_{max}}
\rceil$ we take $h_{max}=0.1, 0.025, 0.0125$, resulting in $L=1,2,4$
respectively.

The results for the $\ell^2(\pazocalbf{L}^2)$-velocity error of
these computations can be seen in Table~\ref{tab:bdf1-gp-nu-dpependence}.
Here we observe that the method is robust with respect to over-stabilisation.
We also note that within the considered range, the method is also robust with 
respect to number of elements in the extension strip. The method seems to be
particularly robust over a wide range of viscosities. With a decrease of the
viscosity by a factor $10^4$, the velocity-error only increased by a factor of
$50$ on the coarsest mesh, while on the finer meshes, this increase was even
smaller. Furthermore, we see that we have a stable solution for very small
$\nu$, where the assumption $\dt\sim\nu$ made in
Theorem~\ref{theorem:ErrorEstimateEnergyNorm}, does not hold anymore.

\begin{table}
    \centering
    \subfloat[][
        $h_{max}=0.1$ and $\dt=0.05$ and a resulting strip-width of $L=1$
    ]{
    \begin{tabular}{ccccccc}
        \toprule
        $\nu \downarrow \backslash \gamma_s \rightarrow$ & 0.1 & 1 & 10 & 100 &
        		1000 \\ 
        \midrule
        1 &\num{4.0237e-02} &\num{4.1476e-02} &\num{4.4530e-02}
			&\num{6.2856e-02} &\num{1.6381e-01} \\
        0.1 &\num{1.8133e-01} &\num{1.8752e-01} &\num{2.3215e-01}
			&\num{4.5091e-01} &\num{1.1653e+00} \\
        0.01 &\num{3.6680e-01} &\num{4.4631e-01} &\num{7.7678e-01}
			&\num{1.3895e+00} &\num{1.5762e+00} \\
        0.001 &\num{4.2229e-01} &\num{6.7312e-01} &\num{1.1210e+00}
			&\num{1.2630e+00} &\num{1.2834e+00} \\
        0.0001 &\num{2.0520e+00} &\num{3.6001e+00} &\num{4.1413e+00}
			&\num{4.3203e+00} &\num{4.3440e+00} \\
        \bottomrule
    \end{tabular}
    }
    \vskip.5\baselineskip
    \subfloat[][
        $h_{max}=0.025$ and $\dt=0.05$ and a resulting strip-width of $L=2$
    ]{
    \begin{tabular}{ccccccc}
        \toprule
        $\nu \downarrow \backslash \gamma_s \rightarrow$ & 0.1 & 1 & 10 & 100 &
        		1000 \\ 
        \midrule
        1 &\num{2.6631e-02} &\num{2.6654e-02} &\num{2.6672e-02}
			&\num{2.6682e-02} &\num{2.6931e-02} \\
        0.1 &\num{1.6878e-01} &\num{1.6880e-01} &\num{1.6893e-01}
			&\num{1.7050e-01} &\num{1.8729e-01} \\
        0.01 &\num{3.2953e-01} &\num{3.3097e-01} &\num{3.3732e-01}
			&\num{3.8196e-01} &\num{6.5119e-01} \\
        0.001 &\num{3.6652e-01} &\num{3.7217e-01} &\num{4.2448e-01}
			&\num{6.7108e-01} &\num{8.9207e-01} \\
        0.0001 &\num{3.4872e-01} &\num{3.4693e-01} &\num{4.3552e-01}
			&\num{5.5976e-01} &\num{5.6543e-01} \\
        \bottomrule
    \end{tabular}
    }
    \vskip.5\baselineskip
    \subfloat[][
        $h_{max}=0.0125$ and $\dt=0.05$ and a resulting strip-width of $L=4$
    ]{
    \begin{tabular}{ccccccc}
        \toprule
        $\nu \downarrow \backslash \gamma_s$ & 0.1 & 1 & 10 & 100 
        		& 1000 \\ 
        \midrule
        1 &\num{2.6040e-02} &\num{2.6045e-02} &\num{2.6048e-02}
			&\num{2.6051e-02} &\num{2.6044e-02} \\
        0.1 &\num{1.6827e-01} &\num{1.6829e-01} &\num{1.6830e-01}
			&\num{1.6837e-01} &\num{1.7009e-01} \\
        0.01 &\num{3.2896e-01} &\num{3.2901e-01} &\num{3.2966e-01}
			&\num{3.3608e-01} &\num{3.9686e-01} \\
        0.001 &\num{3.6628e-01} &\num{3.6533e-01} &\num{3.7471e-01}
			&\num{4.5059e-01} &\num{6.9527e-01} \\
        0.0001 &\num{3.6310e-01} &\num{3.6620e-01} &\num{3.9031e-01}
			&\num{4.7977e-01} &\num{5.3987e-01} \\
        \bottomrule
    \end{tabular}
    }
    \caption{$\ell^2(0,T;\pazocalbf{L}^2(\O(t)))$ velocity error for
        the BDF1 method over a range of viscosities and ghost-penalty 
        parameters.}
    \label{tab:bdf1-gp-nu-dpependence}
\end{table}

\begin{remark}
In computations, where we extended the pressure into the same exterior domain
as the velocity, i.e.  $\TSp$, by using ghost penalties based on the same
facets $\Fhnd$ and using the same ghost-penalty parameter (up to $h$ and
$\nu$-scaling)
$\gamma_{s,u}= \gamma_{s,u}' =\gamma_{s,p}=L \gamma_s$ lead to qualitatively
the same results as the results presented here and in the subsequent sections,
only with slightly larger error constants.
\end{remark}

\subsection{Convergence study}
\label{sec:numerical-examples::subsec:bdf1-convergence}

To investigate the asymptotic convergence behaviour of the method in both time
and space, we compute the \emph{experimental order of convergence} in space
($\text{eoc}_x$) and time ($\text{eoc}_t$), based on the errors of two 
successive refinement levels of the respective variable and the finest
refinement level of the other variable. The analysis predicts that spatial
error is corrupted by a factor $\dt^{-1}$. To investigate this numerically, we
also compute the eoc for one refinement in both space and time
($\text{eoc}_{xt}$).

To study the asymptotic temporal and spatial convergence properties of the
analysed method, we consider the following set-up. The viscosity
is chosen as $\nu=10^{-2}$ and the ghost-penalty parameter is $\gamma_s=1$. The
initial time step is $\dt_0=0.1$ and the initial mesh size is $h_0=0.2$. We
consider a series of uniform refinements by a factor $2$ in both time and
space, such that $h_{max}=h_0\cdot 2^{-L_x}$ and $\dt=\dt_0\cdot2^{-L_t}$. The
remaining parameters are as described in
Section~\ref{sec:numerical-examples::subsec:set-up}. In total we make 4 spatial
and 8 temporal refinements.

The results for all three discrete space-time norms can be seen in
Table~\ref{tab:bdf1-convergence}. With respect to time, we see the expected
linear convergence ($\text{eoc}_t \approx 1$) in all three considered norms.
With respect to space ($\text{eoc}_x$) we see lower than optimal convergence
rates, before the temporal error begins to dominate. However the rates are also
higher than the expected rates if the geometry error was dominant. We attribute
this to an interplay between the geometry error and the
$1/\nu$ scaled consistency error from the ghost-penalties.

With respect to space ($\text{eoc}_x$) we observe at least second order
convergence in all norms. This is as good as the best approximation error for
the $\ell^2(\pazocal{L}^2)$-norm of the pressure and the
$\ell^2(\pazocalbf{H}^2)$-norm of the velocity error while for the 
$\ell^2(\pazocalbf{L}^2)$ velocity error we see a convergence rate which is
worse than its approximation error. This however is to be expected if the
geometry errors with $q=1$ are taken into consideration.

To check whether the factor $(h^{2q} + h^{2m})/\dt$ is observable, we consider
joint refinement of both time and space with $L_t=L_x+4$ for which the the
theory predicts a loss half an order of convergence. However, in
Table~\ref{tab:bdf1-convergence} we see that $\text{eox}_{tx} \approx
\text{eox}_{x}$. This suggests that this part of the analysis is indeed not
sharp, as discussed and expected above.

For the $\ell^2(\pazocal{L}^2)$-pressure error we observe that the experimental
order of convergence in space is higher than expected. This suggests that the
velocity error on the right-hand side of the pressure estimate is dominating
term here.

\begin{table}
    \centering
    \subfloat[][
        $\ell^2(0,T;\pazocalbf{L}^2(\O(t)))$ velocity error.
        \label{tab:BDF1-l2l2v}
    ]{
	\begin{tabular}{ccccccc}
		\toprule
		$L_t \downarrow \backslash L_x \rightarrow$ & 0 & 1 & 2 & 3 & 4 &
			$\text{eoc}_{t}$ \\
		\midrule
		0 &\num{1.2883e+00} &\num{7.4056e-01} &\num{6.5999e-01} &\num{6.4856e-01}
			&\num{6.5360e-01} & -- \\
		1 &\num{1.0210e+00} &\num{4.4098e-01} &\num{3.4341e-01} &\num{3.3113e-01}
			&\num{3.2905e-01} & 0.99 \\
		2 &\num{8.2361e-01} &\num{2.7552e-01} &\num{1.8188e-01} &\num{1.6808e-01}
			&\num{1.6616e-01} & 0.99 \\
		3 &\num{7.3486e-01} &\num{1.9492e-01} &\num{9.8726e-02} &\num{8.5612e-02}
			&\num{8.3691e-02} & 0.99 \\
		4 &\num{6.9090e-01} &\num{1.5642e-01} &\num{5.7087e-02} &\num{4.3802e-02}
			&\num{4.2142e-02} & 0.99 \\
		5 &\num{6.6915e-01} &\num{1.3886e-01} &\num{3.7102e-02} &\num{2.2827e-02}
			&\num{2.1230e-02} & 0.99 \\
		6 &\num{6.5810e-01} &\num{1.3035e-01} &\num{2.7743e-02} &\num{1.2362e-02}
			&\num{1.0740e-02} & 0.98 \\
		7 &\num{6.5260e-01} &\num{1.2638e-01} &\num{2.3462e-02} &\num{7.1958e-03}
			&\num{5.4877e-03} & 0.97 \\
		8 &\num{6.4990e-01} &\num{1.2439e-01} &\num{2.1489e-02} &\num{4.7012e-03}
			&\num{2.8624e-03} & 0.94 \\
		\cmidrule(lr){1-7}
		$\text{eoc}_{x}$ & -- & 2.39 & 2.53 & 2.19 & 0.72 & \\
		$\text{eoc}_{xt}$ & -- & 2.31 & 2.32 & 1.95 & 1.33 & \\
		\bottomrule
	\end{tabular}
    }
    \vskip\baselineskip
    \subfloat[][
        $\ell^2(0,T;\pazocalbf{H}^1(\O(t)))$ velocity error.
        \label{tab:BDF1-l2h1v}
    ]{
	\begin{tabular}{ccccccc}
		\toprule
		$L_t \downarrow \backslash L_x \rightarrow$ & 0 & 1 & 2 & 3 & 4 &
			$\text{eoc}_{t}$ \\
		\midrule
		0 &\num{8.9241e+00} &\num{5.7944e+00} &\num{5.4193e+00} &\num{5.6041e+00}
			&\num{6.2810e+00} & -- \\
		1 &\num{7.8899e+00} &\num{3.9672e+00} &\num{3.2033e+00} &\num{3.2441e+00}
			&\num{3.3660e+00} & 0.90 \\
		2 &\num{6.8992e+00} &\num{2.8372e+00} &\num{1.8868e+00} &\num{1.8398e+00}
			&\num{1.8877e+00} & 0.83 \\
		3 &\num{6.4781e+00} &\num{2.3174e+00} &\num{1.1455e+00} &\num{1.0146e+00}
			&\num{1.0340e+00} & 0.87 \\
		4 &\num{6.2693e+00} &\num{2.0905e+00} &\num{7.5937e-01} &\num{5.5232e-01}
			&\num{5.5041e-01} & 0.91 \\
		5 &\num{6.1709e+00} &\num{2.0002e+00} &\num{5.9228e-01} &\num{3.0337e-01}
			&\num{2.8711e-01} & 0.94 \\
		6 &\num{6.1207e+00} &\num{1.9601e+00} &\num{5.2777e-01} &\num{1.7898e-01}
			&\num{1.4765e-01} & 0.96 \\
		7 &\num{6.0964e+00} &\num{1.9434e+00} &\num{5.0361e-01} &\num{1.2366e-01}
			&\num{7.5832e-02} & 0.96 \\
		8 &\num{6.0849e+00} &\num{1.9355e+00} &\num{4.9410e-01} &\num{1.0245e-01}
			&\num{3.9938e-02} & 0.93 \\
		\cmidrule(lr){1-7}
		$\text{eoc}_{x}$ & -- & 1.65 & 1.97 & 2.27 & 1.36 & \\
		$\text{eoc}_{xt}$ & -- & 1.65 & 1.92 & 2.09 & 1.63 & \\
		\bottomrule
	\end{tabular}
    }
    \vskip\baselineskip
    \subfloat[][
        $\ell^2(0,T;\pazocal{L}^2(\O(t)))$ pressure error.
        \label{tab:BDF1-l2l2p}
    ]{
	\begin{tabular}{ccccccc}
		\toprule
		$L_t \downarrow \backslash L_x \rightarrow$ & 0 & 1 & 2 & 3 & 4 &
			$\text{eoc}_{t}$ \\
		\midrule
		0 &\num{7.2646e-01} &\num{4.4298e-01} &\num{4.1308e-01} &\num{4.2940e-01}
			&\num{4.2461e-01} & -- \\
		1 &\num{5.9283e-01} &\num{2.6494e-01} &\num{2.0480e-01} &\num{1.9921e-01}
			&\num{2.0135e-01} & 1.08 \\
		2 &\num{4.8527e-01} &\num{1.6672e-01} &\num{1.0823e-01} &\num{9.7899e-02}
			&\num{9.6821e-02} & 1.06 \\
		3 &\num{4.4056e-01} &\num{1.1942e-01} &\num{5.9288e-02} &\num{4.9614e-02}
			&\num{4.8036e-02} & 1.01 \\
		4 &\num{4.2658e-01} &\num{9.6590e-02} &\num{3.5067e-02} &\num{2.5428e-02}
			&\num{2.4078e-02} & 1.00 \\
		5 &\num{4.3110e-01} &\num{8.6692e-02} &\num{2.3408e-02} &\num{1.3408e-02}
			&\num{1.2114e-02} & 0.99 \\
		6 &\num{4.5385e-01} &\num{8.3047e-02} &\num{1.7867e-02} &\num{7.4406e-03}
			&\num{6.1440e-03} & 0.98 \\
		7 &\num{5.0066e-01} &\num{8.3405e-02} &\num{1.5421e-02} &\num{4.4924e-03}
			&\num{3.1644e-03} & 0.96 \\
		8 &\num{5.8342e-01} &\num{8.7021e-02} &\num{1.4615e-02} &\num{3.0640e-03}
			&\num{1.6778e-03} & 0.92 \\
		\cmidrule(lr){1-7}
		$\text{eoc}_{x}$ & -- & 2.75 & 2.57 & 2.25 & 0.87 & \\
		$\text{eoc}_{xt}$ & -- & 2.30 & 2.28 & 1.99 & 1.42 & \\
		\bottomrule
	\end{tabular}
    }
    \caption{Mesh-size and time-step convergence for the BDF1 method with
        $\nu=10^{-2}$.}
    \label{tab:bdf1-convergence}
\end{table}

\subsection{Extension to higher order}

\subsubsection{BDF2 time discretisation}
\label{sec:numerical-examples::subsubsec:bdf2-convergence}

As an extension to the presented method, we consider the BDF2 formula to
discretise the time-derivative. To ensure that the appropriate solution history
is available on the necessary elements, we increase the extension strip with
the choice $\delta_h = 2 c_\delta w_n^\infty \dt$. 

We investigate the convergence properties in both time and space. To this end
we take the same basic set up as in
Section~\ref{sec:numerical-examples::subsec:bdf1-convergence}. However we take
$L_x=0,\dots,5$ and $L_t=0,\dots,7$. The results from these computations can be
seen in Table~\ref{tab:bdf2-convergence}.

We observe that with $L_t=7$, the spatial error is dominant on all meshes in
all three norms and that $\text{eoc}_x$ similar to the BDF1 case.

With respect to time, we see second order of convergence ($\text{eoc}_t \approx
2$), while the temporal error is dominant. There are also some stability issues
for large time steps and large $L=16, 32$. However these results are outside
the time-step/viscosity range covered by our theory.

\begin{table}
    \centering
    \subfloat[][
        $\ell^2(0,T;\pazocalbf{L}^2(\O(t)))$ velocity error.
        \label{tab:BDF2-l2l2v}
    ]{
	\begin{tabular}{cccccccc}
		\toprule
		$L_t \downarrow \backslash L_x \rightarrow$ & 0 & 1 & 2 & 3 & 4 & 5 &
			$\text{eoc}_{t}$ \\
		\midrule
		0 &\num{1.3932e+00} &\num{5.6598e-01} &\num{2.4200e-01} &\num{9.9798e-02}
			&\num{1.2444e-01} &\num{1.6709e+00} & -- \\
		1 &\num{1.0988e+00} &\num{2.9700e-01} &\num{1.0703e-01} &\num{3.4729e-02}
			&\num{1.8507e-02} &\num{2.4022e-01} & 2.80 \\
		2 &\num{9.1132e-01} &\num{2.1315e-01} &\num{4.8972e-02} &\num{1.3415e-02}
			&\num{4.6137e-03} &\num{4.0142e-03} & 5.90 \\
		3 &\num{7.6252e-01} &\num{1.6031e-01} &\num{3.3476e-02} &\num{5.7758e-03}
			&\num{1.4420e-03} &\num{9.3571e-04} & 2.10 \\
		4 &\num{7.0052e-01} &\num{1.3991e-01} &\num{2.5888e-02} &\num{4.1726e-03}
			&\num{5.9052e-04} &\num{2.3531e-04} & 1.99 \\
		5 &\num{6.7284e-01} &\num{1.3084e-01} &\num{2.2381e-02} &\num{3.1837e-03}
			&\num{4.5191e-04} &\num{8.3283e-05} & 1.50 \\
		6 &\num{6.6012e-01} &\num{1.2664e-01} &\num{2.0984e-02} &\num{2.8114e-03}
			&\num{3.7218e-04} &\num{6.4434e-05} & 0.37 \\
		7 &\num{6.5380e-01} &\num{1.2450e-01} &\num{2.0324e-02} &\num{2.6568e-03}
			&\num{3.4211e-04} &\num{6.0222e-05} & 0.10 \\
		\cmidrule(lr){1-8}
		$\text{eoc}_{x}$ & -- & 2.39 & 2.61 & 2.94 & 2.96 & 2.51 & \\
		\bottomrule
	\end{tabular}
    }
    \vskip.3\baselineskip
    \subfloat[][
        $\ell^2(0,T;\pazocalbf{H}^1(\O(t)))$ velocity error.
        \label{tab:BDF2-l2h1v}  
    ]{
	\begin{tabular}{cccccccc}
		\toprule
		$L_t \downarrow \backslash L_x \rightarrow$ & 0 & 1 & 2 & 3 & 4 & 5 &
			$\text{eoc}_{t}$ \\
		\midrule
		0 &\num{9.2161e+00} &\num{5.0010e+00} &\num{2.7341e+00} &\num{1.2565e+00}
			&\num{2.5993e+00} &\num{4.2987e+01} & -- \\
		1 &\num{8.2298e+00} &\num{3.3486e+00} &\num{1.5616e+00} &\num{6.0136e-01}
			&\num{2.5936e-01} &\num{1.0491e+01} & 2.03 \\
		2 &\num{7.5469e+00} &\num{2.7354e+00} &\num{8.9500e-01} &\num{3.0682e-01}
			&\num{9.5472e-02} &\num{5.2027e-02} & 7.66 \\
		3 &\num{6.7226e+00} &\num{2.2620e+00} &\num{6.9895e-01} &\num{1.6253e-01}
			&\num{4.4962e-02} &\num{1.5046e-02} & 1.79 \\
		4 &\num{6.3677e+00} &\num{2.0830e+00} &\num{5.9003e-01} &\num{1.3305e-01}
			&\num{2.2486e-02} &\num{5.8579e-03} & 1.36 \\
		5 &\num{6.2153e+00} &\num{2.0022e+00} &\num{5.2983e-01} &\num{1.0712e-01}
			&\num{1.8914e-02} &\num{2.7413e-03} & 1.10 \\
		6 &\num{6.1440e+00} &\num{1.9643e+00} &\num{5.0734e-01} &\num{9.7238e-02}
			&\num{1.5060e-02} &\num{2.3225e-03} & 0.24 \\
		7 &\num{6.1089e+00} &\num{1.9454e+00} &\num{4.9683e-01} &\num{9.3670e-02}
			&\num{1.3500e-02} &\num{1.9016e-03} & 0.29 \\
		\cmidrule(lr){1-8}
		$\text{eoc}_{x}$ & -- & 1.65 & 1.97 & 2.41 & 2.79 & 2.83 & \\
		\bottomrule
	\end{tabular}
    }
    \vskip.3\baselineskip
    \subfloat[][
        $\ell^2(0,T;\pazocal{L}^2(\O(t)))$ pressure error.
        \label{tab:BDF2-l2l2p}
    ]{
	\begin{tabular}{cccccccc}
		\toprule
		$L_t \downarrow \backslash L_x \rightarrow$ & 0 & 1 & 2 & 3 & 4 & 5 &
			$\text{eoc}_{t}$ \\
		\midrule
		0 &\num{8.2044e-01} &\num{3.8054e-01} &\num{2.2387e-01} &\num{1.4999e-01}
			&\num{3.4770e-01} &\num{3.9970e+00} & -- \\
		1 &\num{6.5967e-01} &\num{1.5732e-01} &\num{6.7703e-02} &\num{4.4799e-02}
			&\num{4.0588e-02} &\num{9.5671e-01} & 2.06 \\
		2 &\num{5.6018e-01} &\num{1.2219e-01} &\num{2.3494e-02} &\num{1.0835e-02}
			&\num{1.1464e-02} &\num{1.0501e-02} & 6.51 \\
		3 &\num{4.7649e-01} &\num{9.6691e-02} &\num{1.8724e-02} &\num{2.5424e-03}
			&\num{2.6540e-03} &\num{2.8298e-03} & 1.89 \\
		4 &\num{4.5880e-01} &\num{8.7285e-02} &\num{1.5596e-02} &\num{2.2882e-03}
			&\num{5.0926e-04} &\num{6.8121e-04} & 2.05 \\
		5 &\num{4.7882e-01} &\num{8.4353e-02} &\num{1.4170e-02} &\num{2.0108e-03}
			&\num{2.5893e-04} &\num{1.4580e-04} & 2.22 \\
		6 &\num{5.3120e-01} &\num{8.5391e-02} &\num{1.3851e-02} &\num{1.8687e-03}
			&\num{2.5394e-04} &\num{4.2939e-05} & 1.76 \\
		7 &\num{6.2764e-01} &\num{8.9866e-02} &\num{1.4109e-02} &\num{1.8243e-03}
			&\num{2.4601e-04} &\num{4.1031e-05} & 0.07 \\
		\cmidrule(lr){1-8}
		$\text{eoc}_{x}$ & -- & 2.80 & 2.67 & 2.95 & 2.89 & 2.58 & \\
		\bottomrule
	\end{tabular}
	}
    \caption{Mesh-size and time-step convergence for the BDF2 method with
        $\nu=10^{-2}$.}
    \label{tab:bdf2-convergence}
\end{table}

\subsubsection{Higher-order spatial convergence}
\label{sec:numerical-examples::subsec:higher-order}

As we have seen in the theory and some of the numerical results of the
previous section, the piecewise linear level set approach leads to a loss in
the maximal spatial convergence rate of the velocity. A simple and effective --
though not very efficient -- way to hide the geometrical error and reveal the
underling discretisation error, is the approximation of the the geometry based
on a piecewise linear level set after $s$ subdivisions of cut elements. As the
resulting geometry approximation order is then of order
$\OO\big((\frac{h}{2^s})^2\big)$. The drawback of this approach is a non
optimal scaling for $h\to 0$, since $s$ must increase to balance the
discretisation error on fine meshes.

To balance the geometry error with the discretisation error, we take
$s=\OO\big(\log_2(\nf{1}{h})\big)$ yielding an effective geometry approximation
of $\OO\big(h^{q+1})$ with $q=3$.
Let us stress that the purpose of this ``trick'' is to hide the geometry error
and reveal the underlying remaining discretisation errors and is not meant as a
solution to the problem of approximating unfitted geometries in general.

Using the same set-up as in
Section~\ref{sec:numerical-examples::subsubsec:bdf2-convergence} and the
BDF2 formula to discretise the time-derivative, such that the spatial error is
dominant for larger time-steps, we compute our test-problem over a series of
uniformly refined meshes. The time-step is chosen as $\dt=0.1\cdot 2^{-8}$ and
we consider Taylor-Hood elements with $k=2$ and $k=3$.

The results of these computations can be seen in
Figure~\ref{fig:bdf2-subdiv-convergence}. Here we can see, that we have
recovered optimal order of convergence in both velocity norms. We also note,
that the pressure error converges as before, at an order higher than expected.

\begin{figure}
	\centering
	\includegraphics{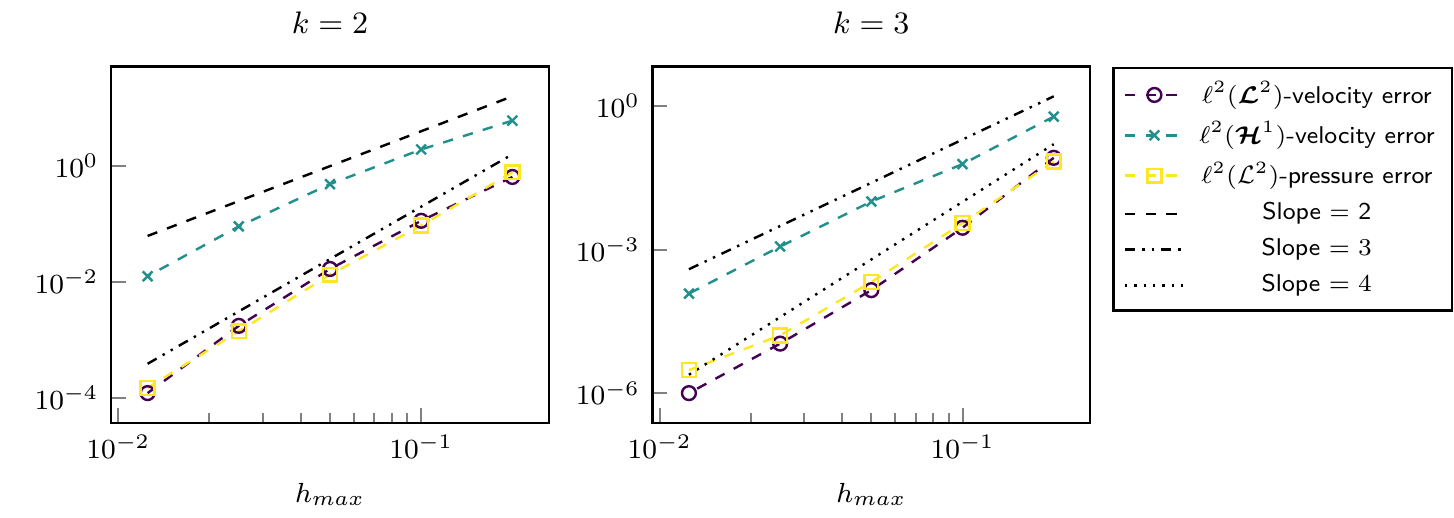}
	\caption{Results with $k=2, 3$ using subdivisions for better approximation of
		the domain boundary.}
	\label{fig:bdf2-subdiv-convergence}
\end{figure}


\section{Conclusions and open problems}
\label{sec:conclusion-open-problems}

We have presented, analysed and implemented a fully Eulerian, inf-sup stable,
unfitted finite element method for the time dependent Stokes problem on
evolving domains. This followed the previous work in \cite{LO19} for
convection-diffusion problems and \cite{BFM19} for the time-dependent Stokes
problem using equal order, pressure stabilised elements. The method is simple
to implement in existing unfitted finite element libraries, since all operators
are standard in unfitted finite elements. 

In the analysis, we have seen that the geometrical consistency error,
introduced by
integrating over discrete, approximated domains $\Onh$ plays a major role and
causes additional coupling between the velocity and pressure errors which is
non-standard. Furthermore, since the time derivative approximation term
$\frac{1}{\dt}(\u^n_h - \u^{n-1}_h)$ is not weakly divergence free with respect
to the pressure space $Q_h^n$, we obtained error estimates which are dependent
on inverse powers of $\dt$. Fortunately, this dependence on $\dt^{-1}$ was not
observable in our numerical results, suggesting that an estimate of
$\nrm{\frac{1}{\dt}(\u^n_h-\u^{n-1}_h)}{-1}$ independent (of negative powers)
of $\dt$ should hold. 

In our numerical experiments, we have also seen, that the geometrical
error---if low order approximations are used---can indeed be a dominating
factor in the final error, corrupting the optimal convergence rate of the
$\PP^2/\PP^1$ finite element pair. We used a simple, but not efficient,
approach for avoiding this in the considered test cases for both $\PP^2/\PP^1$
and $\PP^3/\PP^2$ elements.

Let us now discuss some issues where further refinement of the method and
analysis seem to be of some benefit.

As mentioned above, it seems that it should be possible to prove an estimate
of $\nrm{\frac{1}{\dt}(\u^n_h-\u^{n-1}_h)}{-1}$ independent of $\dt^{-1}$ for
$\u^{n-1}_h$ for which we have $b_h^n(q_h,\u^{n-1}_h)\neq 0$ for some
$q_h\in Q_h^n$.

In Section~\ref{sec:numerical-examples::subsec:higher-order} we have shown that
it is possible to recover the optimal order of convergence for the Taylor-Hood
finite element pair if sufficient computational effort is put on the geometry
approximation. However, the subdivision strategy used here is limited in
its application, due to the large number of subdivisions needed after several
global mesh refinements and for higher order elements, which in turn results
in a large amount of memory needed for computations.

One approach to efficiently obtain higher order geometry approximations for
level set domains has been introduced in \cite{Leh16} for \emph{stationary}
domains. That approaches relies on a slight local deformations of the mesh,
depending on the level set function. For unsteady problems the corresponding
mesh becomes time-dependent for which efficient and accurate transfer
operations are needed from one mesh to the other. This will be discussed in a
forthcoming paper.


\section*{Acknowledgements}
HvW and TR acknowledge support by the Deutsche Forschungsgemeinschaft (DFG,
German Research Foundation) - 314838170, GRK 2297 MathCoRe. TR further
acknowledges supported by the Federal Ministry of Education and Research of
Germany (project number 05M16NMA).
CL gratefully acknowledges funding by the Deutsche Forschungsgemeinschaft (DFG,
German Research Foundation) within the project ``LE 3726/1-1''.
\renewcommand{\bibfont}{\normalfont\footnotesize}
\setlength\bibitemsep{.5\itemsep}
\printbibliography

\end{document}
